\documentclass[10pt, a4paper,leqno]{amsart}
\usepackage[utf8]{inputenc}
\calclayout  % fix the layout of the text
\usepackage{calligra} % Handwriting-style font
\usepackage[T1]{fontenc} % Ensures proper font rendering
\usepackage{color}
\usepackage[english]{babel}
\usepackage[mathscr]{euscript}
\usepackage{amsthm}
\usepackage{thmtools}
\usepackage{stmaryrd}
\usepackage[utf8]{inputenc}
\usepackage[all]{xy}
\usepackage{indentfirst}
\usepackage{hyperref}
\usepackage{amsmath}
\usepackage{amsmath,amsfonts}
\usepackage{lipsum}
\usepackage[toc,page]{appendix}
\usepackage[pdftex]{graphicx}
\usepackage{tikz-cd}
\usepackage{tikz-3dplot}
\usepackage{amsthm}
\usepackage{stmaryrd}
\usepackage{ragged2e}
\usepackage{blindtext}
\usepackage{url}
\usepackage{booktabs} 
\usepackage{array} 
\usepackage{paralist} 
\usepackage{verbatim} 
\usepackage{subfig}
\usepackage{graphicx}			
\usepackage{enumitem}% a package to add options for \begin{itemize}

\theoremstyle{plain}
\newtheorem{Theorem}{Theorem}[section]
\newtheorem{Lemma}[Theorem]{Lemma}
\newtheorem{Proposition}[Theorem]{Proposition}
\newtheorem{Definition}[Theorem]{Definition}
\newtheorem{Remark}[Theorem]{Remark}
\newtheorem{Question}[Theorem]{Question}
\newtheorem{Corollary}[Theorem]{Corollary}
\newtheoremstyle{example}
  {3pt} % Space above
  {3pt} % Space below
  {} % Body font (non-italic)
  {} % Indent amount
  {\bfseries} % Theorem head font
  {.} % Punctuation after theorem head
  {.5em} % Space after theorem head
  {} % Theorem head spec (empty for default)

\theoremstyle{example}
\newtheorem{Example}[Theorem]{Example}

\newtheoremstyle{diagramstyle}
  {3pt} % Space above
  {3pt} % Space below
  {\normalfont} % Body font
  {} % Indent amount
  {\bfseries} % Theorem head font
  {.} % Punctuation after theorem head
  { } % Space after theorem head
  {} % Theorem head spec (can be left empty, meaning `normal`)

\theoremstyle{diagramstyle}

\DeclareMathOperator{\Ass}{Ass}

\DeclareMathOperator{\beg}{indeg}

\DeclareMathOperator{\coker}{Coker}

\DeclareMathOperator{\cohdim}{cohdim}

\DeclareMathOperator{\depth}{depth}
\DeclareMathOperator{\E}{E}

\DeclareMathOperator{\End}{end}
\DeclareMathOperator{\Ext}{Ext}
\DeclareMathOperator{\Fitt}{Fitt}
\DeclareMathOperator{\G}{G}

\DeclareMathOperator{\grade}{grade}

\DeclareMathOperator{\HF}{HF}

\DeclareMathOperator{\Ht}{ht}

\DeclareMathOperator{\Id}{Id}
\DeclareMathOperator{\im}{Im}

\DeclareMathOperator{\Kitt}{Kitt}

\DeclareMathOperator{\pdim}{pdim}

\DeclareMathOperator{\Reg}{reg}

\DeclareMathOperator{\SD}{SD}
\DeclareMathOperator{\SDC}{SDC}

\DeclareMathOperator{\Spec}{Spec}
\DeclareMathOperator{\Supp}{Supp}

\DeclareMathOperator{\Tot}{Tot}

\DeclareMathOperator{\V}{V}

\newcommand{\binomial}[2]{{#1 \choose #2}}
\newcommand{\bt}{{\textbf{t}}}

\newcommand{\fa}{\mathfrak{a}}

\newcommand{\ff}{\mathbf{f}}

\newcommand{\fm}{\mathfrak{m}}

\newcommand{\fp}{\frak{p}}
\newcommand{\fq}{\frak{q}}

\newcommand{\lra}{\longrightarrow}
\newcommand{\pp}{\mathbb{P}}
\newcommand{\ra}{\rightarrow}

\newcommand{\xra}{\xrightarrow}
\newcommand{\xx}{{\textbf{x}}}
\makeatletter
\@namedef{subjclassname@2020}{\textup{2020} Mathematics Subject Classification}
\makeatother

\begin{document}

\title{Set-theoretically perfect ideals and residual intersections}
\author{S. Hamid Hassanzadeh}
\dedicatory{Dedicated to Marc Chardin on the occasion of his 65th birthday}
\email{hamid@im.ufrj.br}
\thanks{The author is partially supported by  (CAPES-Brasil)-Finance Code 001,  grants: CAPES-PrInt Project 88881.311616 and CNPq-Brasil grant number: universal 406377/2021-9.}
\address{Departamento de Matem\'atica, Centro de Tecnologia, Cidade Universit\'aria da Universidade Federal do Rio de Janeiro, 21941-909 Rio de Janeiro, RJ, Brazil}
\date{\today}

\maketitle
%\tableofcontents
\begin{abstract} This paper studies algebraic residual intersections in rings with Serre's condition \( S_{s} \). It demonstrates that  a wide class of residual intersections  is set theoretically perfect.  This fact leads to determining a uniform upper bound for the multiplicity of residual intersections. In positive characteristic, it follows that residual intersections are cohomologically complete intersection and, hence, their variety is connected in codimension one.
\end{abstract}

\section{Introduction}
Residual intersections have a long history in algebraic geometry, tracing back to Cayley-Bacharach theory or at least to the mid-nineteenth century with Chasles \cite{Ch}, who counted the number of conics tangent to a given conic (see Eisenbud's talk \cite{EisT} or Kleiman \cite{K} for a historical introduction). The theory became part of commutative algebra through the work of Artin and Nagata \cite{artin1972residual} where they defined {\it Algebraic Residual Intersections} to study the double point locus of maps between schemes of finite type over a field. The theory of algebraic residual intersection developed alongside the theory of algebraic linkage, with residual intersections being considered a broad generalization of linked ideals. It has advanced significantly thanks to the contributions of C. Huneke, B. Ulrich, A. Kustin, M. Chardin, and D. Eisenbud, among others.

As Eisenbud once said, besides all the miracles that happen in this theory, there are always a bunch of conditions that need to be verified for the theorems to work.
Such conditions have appeared since the work of Artin and Nagata, where they imposed that the ideals have locally few generators-- $G_{s}$-condition. Other conditions have also emerged, such as Strongly Cohen-Macaulay \cite{huneke1983strongly}, Sliding Depth \cite{Slidingdepth}, Sliding depth on powers of ideals \cite{UlrichArtin}, and partial sliding depth \cite{CHU,hassanzadeh2012cohen}.

Cohen-Macaulayness has always been a central property for residual intersections since the work of Peskine-Szpiro \cite{PS}. This is a reason why all of the above conditions appear all the time.

In this paper, we study residual intersections of ideals that do not possess any special homological properties. However, their number of generators is at most the height of the desired residual intersection and it does not decrease after localization. These types of ideals are called \( r \)-minimally generated ideals. More specifically, \( r \)-minimally generated ideals are situated on the opposite side of ideals with the \( G_{s} \) condition. While the \( G_{s} \) condition forces the local number of generators of an ideal to be at most the dimension of the ring, \( r \)-minimally generated ideals maintain the same minimal number of generators, \( r \), from a certain codimension onward.

Although  the residual intersections of \( r \)-minimally generated ideals is not necessarily Cohen-Macaulay, they are locally ``set theoretically perfect'' when the base ring is Cohen-Macaulay.  

That in regular rings, every module admits a finite free resolution,  is crucial in the development of the theory of multiplicity initiated by Hilbert.  However, when a variety is embedded into another variety, we often lose the finiteness of the free resolution of the defining ideal. To overcome this problem, many advanced techniques have been developed in commutative algebra. Some instances are {\it rationality of Poincare series} or invariants such as {\it Castelnuovo-Mumford regularity and rate}.

 The concept of  ``free approach'', \autoref{DFA},  is another attempt to overcome the deficiency of the homological methods to study algebraic sets inside an arbitrary variety. 
 A leitmotiv of this paper is to 
 \begin{center}
 {\it Determine ideals in a Noetherian ring which admit a free approach. }
 \end{center}
 
 The first examples of such ideals are {\it set-theoretic complete intersection} ideals. {\it  Equimultiple ideals} are instances of set-theoretic complete intersection. It is worth mentioning that, Hartshorne's question which inquires whether every irreducible curve in $\mathbb{P}^{3}$ is a set-theoretic complete intersection is still open in characteristic zero, \cite[Exercise 2.17]{Hartshorne}.  
 The idea of  free approach suggests the following strategy to provide an affirmative answer (if one exists) to Hartshorne's problem:

\begin{Proposition}\label{Hartshorne}  Every closed irreducible curve in $\mathbb{P}^{3}$ is a set-theoretic intersection of two surfaces if and only if in $R=k[x_{0},x_{1},x_{2},x_{3}]$
\begin{itemize}
\item{For every homogenous ideal  $I$ which is Cohen-Macaulay of codimension $2$  and its radical is prime, there exist $f,g\in R$ such that ${\sqrt I}=\sqrt {(f,g)}$, and}
\item{Every codimension two prime ideal in $R$ admits a free approach.}
\end{itemize}
\end{Proposition}

 We hope the idea of free approach helps to find a homological method to search set theretically  complete intersections.

 The resume of the  main results of this paper is the following.
 \begin{Theorem} Let $R$ be a Noetherian ring that satisfies the Serre's condition S$_{s}$, $I$ an $r$-generated  ideal of $R$,  $\fa$ an $s$-generated sub-ideal of $I$ and $J=\fa:I$. Assume $I$ contains a regular element and $\Ht(J)\geq s$. If any of the following conditions holds then $J$ admits a free approach, i.e. there is an ideal $\tau\subseteq J$ which has finite projective dimension and ${\sqrt \tau}={\sqrt J}$.
 \begin{itemize}
 \item $s\geq r$ and $I_{\fp}$ is minimally generated by $r$ elements for any prime $\fp\supseteq I$ of height at least  $s-1$ ; or 
 \item $R$ is local with infinite residue field, $I$ is generated by a proper sequence and $\pdim(Z_{i})\leq r-i-1$ for all $i\geq 1$, where $Z_{i}$ is the $i$th Koszul cycle of the generators of $I$.
 \end{itemize} 
 \end{Theorem}
As a consequence of the first part of the above theorem, we show in \autoref{ericci} that the multiplicity of the residual intersection of a complete intersection with the same degrees serves as a uniform upper bound for the multiplicity of residual intersections of \( r \)-minimally generated ideals. This advances our understanding of the multiplicity of residual intersections, as explained in \cite{CHU}.

A consequence of the second part is \autoref{regularring}: Let \( R \) be a regular local ring with infinite residue field and the ideal  \( I \) satisfies  \(\SD\) and G$_s^-$. Then any algebraic residual intersection of \( I \) admits a free approach. Moreover, any arithmetic \( s \)-residual intersection of \( I \) is Cohen-Macaulay. This advances the answer to  the last unsolved part of the question posed by Huneke and Ulrich \cite[Question 5.7]{huneke1988residual}:
\begin{center}
 \textit{
 Whether in a Cohen-Macaulay local ring, the algebraic residual intersections of ideals with sliding depth are Cohen-Macaulay?
 }
 \end{center}

 %\newpage
 %%%%%%%%%%%-----------------------------------------------------------------first section-------------------------------------------------------------------------
\section{Free approach and $r$-minimally generated ideals}
The rings are assumed to be commutative with identity and of finite Krull dimension.
\subsection{free approach}

%====def 
 \begin{Definition}\label{DFA} Let $R$ be a ring and $J$ an ideal of $R$. A {\bf free approach} to $J$ is an acyclic  bounded complex of finitely generated free $R$-modules, $F_{\bullet}: 0\ra F_{s}\ra\ldots \ra F_{1}\ra F_{0}\to 0$,  such that $s=\Ht(J)$, $F_{0}=R$, $H_{0}(F_{\bullet})= R/\tau$, $\tau\subseteq J$ and $\sqrt\tau=\sqrt J.$ 
 \end{Definition}
 
 When $R$ is a Cohen-Macaulay local ring, the ideal $\tau$ in the above definition is a perfect ideal, meaning it is a Cohen-Macaulay ideal with finite projective dimension. Consequently, one can describe $J$ in \autoref{DFA} as set-theoretically Cohen-Macaulay or set-theoretically perfect. The notion of set-theoretically Cohen-Macaulay ideals in a regular local ring was first introduced in \cite[page 599]{EisMustStillman}. Singh and Walter in \cite{SinghWalter} show that for a smooth elliptic curve $E_{\mathbb{Q}}\subset \pp^2_{\mathbb{Q}}$, the defining ideal of
$E_{\mathbb{Q}}\times \pp^1_{\mathbb{Q}}\subset \pp^5_{\mathbb{Q}}$ is not set-theoretically Cohen-Macaulay.
 
 We observe that when $R$ is not necessarily Cohen-Macaulay, any set-theoretically perfect ideal admits a free approach. However, the converse does not hold in general. 
 
When $J$ is generated by a regular sequence,  the Koszul complex is the free resolution of $R/J$. Thence, a free approach can be applied to any ideal whose radical is the same as the radical of $J$. That is, \textit{set-theoretic complete intersections} admit free approaches. In the case of characteristic $p>0$, this class contains all ideals of height $n-1$ in the ring $R=k[x_{1},\ldots,x_{n}]$ or $R=k[[x_{1},\ldots,x_{n}]]$, where $k$ is a perfect field \cite{CowiskNori}. A famous question posed by R. Hartshorne \cite[Exercise 2.17]{Hartshorne} asserts that ``every closed irreducible curve in $\mathbb{P}^{3}$ is a set-theoretic intersection of two surfaces.'' This problem is known to have an affirmative answer for several families of monomial curves in positive characteristics and some special rational monomial curves in characteristic zero, including Cohen-Macaulay ones \cite{Hartshornecharp, Ferrand, Thoma, RobbianoValla}. The problem remains widely open in characteristic zero and for non-monomial curves.

%In a Cohen-Macaulay local ring $R$, any perfect ideal $J$ (i.e., $R/J$ is Cohen-Macaulay of finite projective dimension) admits a free approach. Indeed, the minimal free resolution of $R/J$ is a free approach to $J$.

Unlike free resolutions, not every ideal in every ring admits a free approach. \autoref{PCM}  follows from the {\it new intersection theorem}   \cite{FoxbyII, HochsterTopics,RobertsIntersection}.  See also \cite[Corollary 9.4.2]{bruns1998cohen} for further details.
%=====prop
\begin{Proposition}\label{PCM} Let $R$ be a Noetherian ring. If  $J$ is an ideal of height $s$ that  admits a free approach then
\begin{enumerate}
\item[(i)] $\Ht(\fp)=s$ for any $\fp\in \min(J)$,
 \item[(ii)] $\depth(R_{\fp})\geq  s$  for any prime ideal $\fp\supseteq J$. 
 \item[(iii)] $\grade(J)=\Ht(J)=s$
\end{enumerate}
In particular, $R_{\fp}$ is Cohen-Macaulay for any $\fp\in \min(J)$.
\end{Proposition}
\begin{proof} We keep the notation in \autoref{DFA}. Let $\fp\in \min(J)$. Since $\sqrt\tau=\sqrt J$, $\tau_{\fp}$ is a $\fp R_{\fp}$-primary ideal with finite free resoltion of length $s$. The new intersection theorem implies that $s\geq \dim(R_{\fp})$.  Since $\Ht(\fp)\geq \Ht(J)=s$ part $(i)$ follows.

$(ii)$. Let $\fp\supseteq J$ be a prime ideal. Then $(R/\tau)_{\fp}$ is a non-zero finitely generated module over Noetherian local ring $R_{\fp}$ of finite projective dimension $\leq s$. Since, by part (1), all minimal primes of $\tau_{\fp}$ have the same height $s$, another application of the new intersection theorem implies that $\pdim(R/\tau)_{\fp}=s$. By Auslander-Buchsbaum formula, $\pdim(R/\tau)_{\fp}+\depth (R/\tau)_{\fp}=\depth R_{\fp} $ that complete the proof of part (2).

$(iii)$.  According to $(i)$ and $(iii)$, for all prime ideal $\fp\supseteq J$, we have $\Ht(\fp)\geq s$ and $\depth(R_{\fp})\geq s$. On the other hand, $\grade(J)=\min\{\depth(R_{\fp}):\fp\in \V(J)\}$, \cite[1.2.10]{bruns1998cohen}. Hence $\grade(J)\geq s$. On the other hand $\grade(J)\leq \Ht(J)=s$.
\end{proof}

\autoref{PCM} provides a criterion to decide when an ideal does not admit a free approach. One might wonder if every ideal admits a free approach in a well-structured ring, such as a Cohen-Macaulay ring.  The next result shows that at least in positive characteristics only special ideals may admit a free approach. We recall two definitions.
\begin{Definition}\label{cohdim} Let $R$ be a Noetherian ring and $\fa$ an ideal of $R$. The {\bf cohomological dimension} of $\fa$ (with respect to $R$) is $$\cohdim(\fa):=\max\{i\in \mathbb{N}:H^{i}_{\fa}(R)\neq 0\}.$$
An ideal $\fa$ is called {\bf cohomologically complete  intersection} whenever $\cohdim(\fa)=\grade(\fa)$. 
\end{Definition}
The other definition is 
\begin{Definition}\label{c(R)} Let $R$ be a Noetherian ring. The {\bf connectedness dimension} of $R$ is the least integer $i$ such that by removing any variety of dimension $i-1$, $\Spec(R)$ maintains its connectivity. This number is denoted by $c(R)$. When we say that $R$ is connected in dimension $j$, it implies $j\leq c(R)$. We refer to \cite[Chapter 19]{localCohomology} for details. 
\end{Definition}
The idea of the proof of the following proposition comes from \cite[Proposition 4.1, page 386]{finiteprojectivepeskine} and also \cite[Theorem 8.2.6]{bruns1998cohen}.
%%=============================================prop
\begin{Proposition}\label{Pcohdim} Let $R$ be a Noetherian ring with positive characteristic  and of dimension $d$.  Let $J$ be an ideal of $R$ of height $s$ which admits a free approach. Then 
\begin{itemize} 
\item[(i)]{$J$ is  cohomologically complete intersection.}
\item[(ii)]{If $\hat{R}$ is connected in dimension $d-c$, then $R/J$ is connected in dimension $\min\{d-s-c,d-s-1\}$.}
\end{itemize}
\end{Proposition}

\begin{proof}  Let    $$F_{\bullet}: 0\ra F_{s}\xrightarrow{\phi_{s}}\ldots \ra F_{1}\xrightarrow{\phi_{1}} F_{0}\to 0$$  be the free  approach of $J$ with $H_{0}(F_{\bullet})= R/\tau$, $\tau\subseteq J$ and $\sqrt\tau=\sqrt J.$

Since $F_{\bullet}$ is an acyclic complex of finite free modules, the Buchsbaum-Einsenbud criterion implies that $\grade(I_{r_{i}}(\phi_{i}))\geq i$ where
$r_i =\sum_{j=0}^{s-i}(-1)^j{\rm rank}(F_{i+j})$, for all $i$. Set $p>0$ to be the characteristic of $R$. Applying the Frobenius functor $\mathcal{F}$ on $F_{\bullet}$, the maps lift to their $p$'th power, hence  $\grade(I_{r_{i}}(\mathcal{F}(\phi_{i}))=I_{r_{i}}(\phi_{i})^{[p]})\geq i$. Thus, again the Buchsbaum-Einsenbud criterion implies that $\mathcal{F}(F_{\bullet})$ is acyclic. The bonus is that $H_{0}(\mathcal{F}(F_{\bullet}))= R/\tau^{[p]}.$ By iterating this procedure one gets that for all integer $e$, $\pdim(R/(\tau^{[p^{e}]}))\leq s$.
  In particular, $$\Ext^{j}_{R}(R/\tau^{[p^{e}]},R)=0\quad\text{for~all~} j> s.$$
  On the other hand $\grade(\tau^{[p^{e}]})=\grade(\tau)=\grade(J)=s$ by \autoref{PCM}(iii), hence 
$$\Ext^{j}_{R}(R/\tau^{[p^{e}]},R)=0\quad\text{for~all~} j< s.$$
 Taking the direct limit of the family of Ext-modules, we have  for all $j\neq s$ 
 $$H^{j}_{J}(R)=H^{j}_{\tau}(R)=\varinjlim_{e}\Ext^{j}_{R}(R/\tau^{[p^{e}]},R)=0.$$
   $(ii)$ is a consequence of $(i)$ according to \cite[Theorem 2.8]{DivTousi}. We have  $c(R/J)\geq c(\widehat{R/J})$.  If $c(\widehat{R})\geq 1$ then \cite[Theorem 2.8]{DivTousi} implies that $c(R/J)\geq \dim(\widehat{R/J})-c$ and if $c(\widehat{R})=-1$ or $0$, then $c(R/J)\geq \dim(\widehat{R/J})-1$.

\end{proof}

%===================================================

The following exemplifies ideals that do not admit free approaches. 
\begin{Example}Let $k$ be a field, $R=k[x_{1},\ldots,x_{4}]$ and $J=(x_{1},x_{2})\cap(x_{3},x_{4})$. Applying the Mayer-Vietories exact sequence, one sees that, independent of the characteristic of the field,   $\cohdim(J)=3$. In particular, if $k$ has a positive characteristic then $J$ does not admit a free approach according to \autoref{Pcohdim}(i). In \cite[Theorem-Example 5.6]{LyubHuneke}, the authors provide an example of a prime ideal that is not a cohomologically complete intersection, hence not admitting a free approach. 

On the positive side,  after proving \autoref{residualsliding}, we see that in a regular local ring residual intersections of ideals with sliding depth admit free approach. \end{Example}

\subsection{$r$-minimally generated ideals}

The theory of residual intersections emerged simultaneously with the introduction of the $G_{s}$-condition \cite{artin1972residual}. For an ideal $I$ in a Noetherian ring $R$, $I$ is said to satisfy the $G_{s}$-condition if $\mu(I_{\fp}) \leq \Ht(\fp)$ for all prime ideals $\fp \supset I$ with $\Ht(\fp) \leq s-1$.  ($\mu(-)$ is the function of the minimal number of generators). The $G_{s}$-condition is instrumental in reducing questions about residual intersections to those in linkage theory. Additionally, in the arithmetic of blow-up algebras, this condition is both necessary and sufficient for many theorems \cite{vasconcelos1994arithmetic}.

Alternatively, {\it $r$-minimally generated ideals} are located on the oposite side of ideals with $G_{s}$ condition. While the $G_{s}$ condition forces the local number of generators of an ideal to be at most the dimension of the ring, $r$-minimally generated ideals keep the same minimal number of generators, $r,$ from a certain codimension onward.
We  see in  our main result, \autoref{residual1}, that residual intersections of $r$-minimally generated ideals admit free approaches.  The remarkable point about this result is that such residual intersections are far from being Cohen-Macaulay. Hence, a priori, it is not obvious that they admit a free approach.

%==def r-min
\begin{Definition}\label{Drmin} Let $R$ be a Noetherian  ring and $I$ an ideal of $R$. Let $\zeta$ be an integer. We say that $I$ is {\bf $r$-minimally generated from height $\zeta$}, if $\mu(I_{\fp})=r$  for any prime ideal $\fp\supseteq I$with $\Ht(\fp)\geq \zeta$. In the case where $\zeta=\Ht(I)$, we say $I$ is $r$-minimally generated.
\end{Definition}
Clearly, any ideal in a local ring is $r$-minimally generated from the height equal to the dimension of the ring. First examples of $r$-minimaly generated ideals which are not complete intersections, are almost complete intersection prime ideals which are NOT generically complete intersections. For instance, the defining ideal of a singular point of a hypersurface(in $\mathbb{A}^{n}$ or $\mathbb{P}^{n}$) is $n$-minimally generated.  At present,  the current theory of algebraic residual intersections does not address the study of subschemes of a hypersurface-minus-a-singular point.

Observe, $r$-minimally generated ideals exist for any number of generators $r$ and  any  height $h$. To see this, let $I$ be a complete intersection ideal generated by $r$ elements in a Noetherian ring $R$.  Let $h=r-c$ be the desired height. Take $f_{1},\ldots, f_{c}\in I$  a regular sequence. Then the ideal $I/(f_{1}^{2},\ldots,f_{c}^{2})$ in $R/(f_{1}^{2},\ldots,f_{c}^{2})$ is $r$-minimally generated of height $h$. 

Using Fitting ideals, one can easily obtain a criterion to decide when an ideal is $r$-minimally generated.  Let $j \in \mathbb{Z}$ be an integer. Set $\V_{j}(I):=\{\fp\in \Spec(R): \fp\supseteq I, \Ht(\fp)=j\}$. 
\begin{Proposition}\label{PFitt} Let $R$ be a Noetherian ring, $I$ an ideal minimally generated by $r$-elements and $R^{m}\xrightarrow{\phi}R^{r}\to I\to 0$  a presentation of $I$. In the following, except for part (i), we assume $r\geq 2$. 
\begin{itemize}
\item[(i)]{ For any integer $\zeta$,  $\bigcup_{j\geq \zeta}\V_{j}(I)\subseteq \V(I_{1}(\phi))$ if and only if $I$ is $r$-minimally generated from height $\zeta$.}
\item[(ii)]{ $\V(I_{1}(\phi))\subseteq \V(I)$, in particular $\Ht(I_{1}(\phi))\geq \Ht(I)$. }
\item[(iii)] Let $\zeta=\Ht(I_{1}(\phi))$. If $R$ is  local  and  $\dim(R)>\zeta>\Ht(I)$ then $I$ is not $r$-minimally generated from height $\zeta$.
\item[(iv)] Suppose that $\zeta:=\Ht(I_{1}(\phi))=\Ht(I)$. If  $I$ is $r$-minimally generated then $\V_{\zeta}(I_{1}(\phi))=\V_{\zeta}(I)$. Conversely, if $\min(I)= \V_{\zeta}(I)$ and this set has the same cardinality as $\V_{\zeta}(I_{1}(\phi))$ then $I$ is $r$-minimally generated.
\end{itemize}
\end{Proposition}
\begin{proof} $(i)$. This part follows from the facts that   $I_{1}(\phi)=\Fitt_{r-1}(I)$ and  $\V(\Fitt_{r-1}(I))=\{\fp\in \Spec(R): \mu(I_{\fp})\geq r\}$, \cite[Proposition 20.6]{EisenbudBook}.

$(ii)$. Let $\fp\supseteq I_{1}(\phi)$, tensoring the exact sequence $R^{m}\xrightarrow{\phi}R^{r}\to I\to 0$ into $R_{\fp}/\fp R_{\fp}$, we get $\dim(I_{\fp}/\fp I_{\fp})=r\geq 2$, as a vector space. On the other hand, if $\fp\not\supseteq  I$, then $I_{\fp}/\fp I_{\fp}=R_{\fp}/\fp R_{\fp}$ is a vector space of dimension $1$ which is a contradiction. Hence $\fp\supseteq I$.

$(iii)$.  The set $\V_{\zeta}(I_{1}(\phi))=\min(I_{1}(\phi))$ is a finite set. Let $\fm$ be the maximal ideal of $R$. Since $\dim(R)>\zeta>\Ht(I)$,  and $\fm\supset \sqrt{I_{1}(\phi)}\supset \sqrt{I}$, we have  $\Ht(\fm/I)>1$. Therefore Ratliff's weak existence theorem \cite[Thorem 31.2]{matsumura1989commutative} implies that $\V_{\zeta}(I)$ is an infinite set, in particular $\V_{\zeta}(I)\not\subseteq V(I_{1}(\phi))$, so that $I$ is not $r$-minimally generated from height $\zeta$, by part $(i)$.

$(iv)$. By $(ii)$, $\V_{\zeta}(I_{1}(\phi))\subseteq \V_{\zeta}(I)$.  If $I$ is $r$-minimally generated then by $(i)$, $\V_{\zeta}(I)\subseteq \V_{\zeta}(I_{1}(\phi)) $, hence $\V_{\zeta}(I_{1}(\phi))=\V_{\zeta}(I)$.  Conversely, if $\V_{\zeta}(I)$ has the same cardinality as $\V_{\zeta}(I_{1}(\phi))$ then $\V_{\zeta}(I_{1}(\phi))=\V_{\zeta}(I)=\min(I)$. On the other hand for any $j>\zeta$, any prime ideal $\fp$ of height $j$ that contains $I$ contains a minimal prime of $I$, that is $\fp\in  \V(I_{1}(\phi))$. Hence by $(i)$, $I$ is $r$-minimally generated. 
\end{proof}

We continue with our working example.
\begin{Example}\label{Ex2} Let $k$ be a field. 
\begin{itemize}
\item[(i)]{Set  $A= k[x_{1},\ldots,x_{4}]$ an $^{*}$local ring and $R=A/(x_{1},x_{2})\cap(x_{3},x_{4})$. Let $I=(x_{1},x_{2},x_{3},x_{4})\subset R$. $I$ is a prime ideal of height $2$. It is easy to see that $I=I_{1}(\phi)$. Hence, according to \autoref{PFitt}(iv), $I$ is $4$-minimally generated from height $2$. }
\item[(ii)]{ Set $A=k[x_{1},\ldots,x_{6}]$ and $R=A/(x_{1}x_{2}x_{3}x_{4},x_{2}x_{3}x_{4}x_{5},x_{1}x_{2}x_{3}x_{6},x_{1}x_{2}x_{5}x_{6},x_{1}x_{4}x_{5}x_{6},x_{3}x_{4}x_{5}x_{6})$ (This example appears in \cite[Example 3.6]{MuraiTerai}, as an example of a Stanley-Reisner ring which is S$_{3}$ but not Cohen-Macaulay). Let $H$ be a random $3\times 2$ matrix of linear forms in $R$, for instance 
$$H=\left(
\begin{matrix} 
x_1+x_6&x_1+x_2+x_5+x_{6}\\
x_{3}+x_4+x_5&x_1+x_2+x_3\\
x_2+x_{4}+x_5+x_{6}&x_1+x_3+x_4+x_5\\
\end{matrix}
\right).
$$
Put $I=I_{2}(H)$. Then $I$ is an ideal of codimension $2$ and projective dimension $1$. It is not difficult to see that the ideal generated by entries of $H$ is of codimension $4$. Accordingly, \autoref{PFitt}(i) implies that  $I$ is not $3$-minimally generated from height $3$.

}
\end{itemize}
\end{Example}

%We will see, as a corollary of \autoref{residualsliding}, that the above example is a special case of a family of almost complete intersection ideals which are  not  $r$-minimally generated. 
%%%%%%%%%%%%%%
\section{Residual Intersections}\label{ResidualSection}

This section consists of two parts. In the first part, we focus on $s$-residual intersections  $J=(\fa:I)$ where $I$ is an $r$-minimally generated ideal and $s\geq r$.
To our knowledge, this type of residual intersections has yet to be studied in the literature.   A prototype example of such residual intersections is an $r$-residual intersection of an almost complete intersection prime ideal $I$ generated by $r$ elements, assuming, in addition, that $I$ is not generically a complete intersection. 
In the second part, we replace the conditions $s\geq r$ and $r$-minimally generation by assuming that  the ideal $I$ satisfies the sliding depth condition.  

To study residual intersections, we follow the method introduced in \cite{hassanzadeh2012cohen} and further developed in the works \cite{HamidNaeliton},\cite{ChardinNaeliton} and \cite{boucca2019residual}.

We first recall the basic definition of residual intersections. 
\begin{Definition}
Let 
 $R$ be  a Noetherian ring of dimension $d$.
\begin{itemize}
\item{an {\it (algebraic) $s$-residual intersection} of $I$ is a proper ideal $J$ of $R$ such that $\Ht(J)\geq s$ and $J=(\fa:_R I)$ for some ideal $\fa\subset I$  generated by $s$ elements.}
\item{A {\it geometric $s$-residual intersection} of $I$ is an algebraic $s$-residual intersection $J$ of $I$ such that $\Ht(I+J)\geq s+1.$}
\item An {\it arithmetic $s$-residual intersection}  $J=(\fa:I)$ is an algebraic $s$-residual intersection such that $\mu_{R_{\fp}}((I/\fa)_{\fp})\leq 1$ for all prime ideal $\fp \supseteq (I+J)$ with $\Ht(\fp)\leq s$. ( $\mu$ denotes the minimum number of generators.)
\end{itemize}

\end{Definition}
Throughout $R$ is   a Noetherian ring of dimension
$d$,  $I=(f_1,\ldots,f_r)$ is an ideal of grade $g\geq1$,\footnote {The condition $g\geq1$ is necessary to construct complexes of length $s$ rathar than $s+1$. In most of the proofs, the fact that $g\geq1$ does not appear.}
$\fa=(a_1,\ldots,a_s)$ is an ideal contained in $I$, $
J=\fa:_{R}I$--not necessarily an $s$-residual intersection--
and $S=R[t_1,\ldots,t_r]$ is a standard polynomial extension
of $R$ with indeterminates $t_{i}$'s , i.e.  $\deg(t_i)=1$ for all
$i$ and  $S_0=R$.  We set $\bt=t_1,\ldots,t_r$. For a graded
module $M$, $\beg(M):=\inf\{i : M_i\neq 0\}$ and
$\End(M):=\sup\{i : M_i\neq 0\}$.  For a sequence of elements $\xx$
in a commutative ring $A$ and an $A$-module $M$, we denote the
Koszul complex by $K_\bullet(\xx;M)$, its cycles by
$Z_i(\xx;M)$ and homologies by $H_i(\xx;M).$

 Let 
$$(a_{1}\ldots a_{s})=(f_{1}\ldots f_{r})
\left(\begin{matrix}
c_{11}&\ldots&c_{1s}\\
\vdots&\vdots&\vdots\\
 c_{r1}&\ldots&c_{rs}\\
\end{matrix}\right)
\quad \text{and }\quad 
(\gamma_{1}\ldots \gamma_{s})=(t_{1}\ldots t_{r})
\left(\begin{matrix}
c_{11}&\ldots&c_{1s}\\
\vdots&\vdots&\vdots\\
 c_{r1}&\ldots&c_{rs}\\
\end{matrix}\right).
$$
 Consider the Koszul complex
 \begin{center}$K_\bullet(f;R):
0\rightarrow K_r \xrightarrow{\delta_{r}^{f}} K_{r-1}
\xrightarrow{\delta_{r-1}^f} \ldots \rightarrow K_0\rightarrow 0.$
\end{center}

Simis-Vanconcelos and Herzog-Simis-Vasconcelos introduced the approximation complexes \cite{Approximation}. The  $\mathcal{Z}$-complex is crucial in studying residual intersection in \cite{hassanzadeh2012cohen}.  The  $\mathcal{Z}$-complex of $I$ with coefficients in $R$  is the following complex 

\begin{center}
$\mathcal{Z}_\bullet=\mathcal{Z}_\bullet(f;R): 0\rightarrow
Z_{r-1}\bigotimes_{R}S(-r+1) \xrightarrow{\delta_{r-1}^T} \ldots
\rightarrow Z_1\bigotimes_{R}S(-1) \xrightarrow{\delta_{1}^T}
Z_0\bigotimes_{R}S\rightarrow 0,$
\end{center}
where $Z_i=Z_i(f;R)$. 
Notice that since $\grade(I)>0$, hence $Z_r=0$.   The map $\delta^{T}_{i}: Z_{i}\bigotimes_{R}S(-i)\to Z_{i-1}\bigotimes_{R}S(-i+1)$ is defined by  sending $$e_{j_{1}}\wedge\ldots\wedge e_{j_i}\otimes P(\bt)\mapsto \sum_k(-1)^ke_{j_{1}}\wedge\widehat{e_{j_k}}\wedge \ldots\wedge e_{j_i}\otimes t_kP(\bt).$$
 A significant property of this complex is that $H_0(\mathcal{Z}_\bullet)$ is the symmetric algebra of $I$, denoted by $\mathcal{S}_I$. 
 $H_i(\mathcal{Z}_\bullet)$ is finitely generated $\mathcal{S}_I$-module for all $i$ \cite[4.3]{Approximation}. 
 In  \autoref{section3.2} we review more properties of this complex. For further details about approximation complexes see \cite{Approximation} and \cite{vasconcelos1994arithmetic}.

%%%%%%%%%%%
\subsection{residual intersections of $r$-minimally generated ideals} Our goal in this section is to construct a free approach for the residual intersection of $r$-minimally generated ideals. We modify the residual approximation complexes introduced in \cite{hassanzadeh2012cohen} to this order.

The maps $\delta^{T}_{i}:Z_{i}\bigotimes_{R}S(-i)\to Z_{i-1}\bigotimes_{R}S(-i+1)$ lift naturally to the maps
$\bar \delta^{T}_{i}:K_{i+1}\bigotimes_{R}S(-i)\to K_{i}\bigotimes_{R}S(-i+1)$. Using these extended maps, we construct a double complex:

\begin{equation}\label{black1}
 \xymatrix{
         &                           &                                                                               &             & K_{r}\otimes S(-1) \ar[d] &                       \\
         &                           &                                                                               &0\ar[r]      &K_{r-1}\otimes S(-1) \ar[d] &    \\
         &                           &  0 \ar[d]                                                                     &      &\vdots \ar[d]  &                   \\
         &  0 \ar[r] \ar[d]          & K_r\otimes S(-r+2)\ar[r] \ar^{\delta_{r}^f \otimes 1}[d] &\ldots\ar[r] &K_{3}\otimes S(-1) \ar^{\delta_{3}^f \otimes 1}[d]& \\
 0 \ar[r]& K_r\otimes S(-r+1) \ar^{\bar \delta^{T}_{r}}[r]\ar^{\bar \delta_{r}^f \otimes 1}[d] & K_{r-1}\otimes S(-r+2) \ar^(.7){\bar \delta^{T}_{r-1}}[r] \ar^{\bar \delta_{r-1}^f \otimes 1}[d]&\ldots\ar^{\bar \delta^{T}_{3}}[r] &K_{2}\otimes S(-1)\ar^{\bar \delta_{2}^f \otimes 1}[d]&        \\
 \textcolor{orange}{0}\ar@[yellow][r]& \textcolor{orange}{Z_{r-1}\bigotimes_{R}S(-r+1)}\ar@[yellow]^{ \delta^{T}_{r-1}}[r]& \textcolor{orange}{Z_{r-2}\otimes S(-r+2)} \ar@[yellow]^(.7){ \delta^{T}_{r-2}}[r] & \ldots\ar@[yellow]^{\delta^{T}_{2}}[r] & \textcolor{orange}{Z_1\bigotimes_{R}S(-1)}\ar@[yellow]^{\delta^{T}_{1}}[r]&  \textcolor{orange}{R\bigotimes_{R}S}\\ 
}
\end{equation}
The horizontal lines in the black part of the double complex in \autoref{black1} are the Koszul complex with respect to $\bt$.  We totalize the black part of \autoref{black1} to construct a complex of finite free $S$-modules $0\ra L_{r-1}\ra\ldots\ra L_{2}\xrightarrow{(\delta_{3}^f \otimes 1)\oplus \delta^{T}_{3}} L_{1}$ where $L_{i}=K_{i+1}\otimes_{R}(\oplus_{j=1}^{i}S(-j))$ for $i\geq 1$. Augmenting this complex with another term, we get the following complex of finite free $S$-modules: 
\begin{equation}\label{Lcomplex}
L_{\bullet}:0\longrightarrow L_{r-1}\longrightarrow\ldots\longrightarrow L_{2}\xrightarrow{(\delta_{3}^f \otimes 1)\oplus \delta^{T}_{3}} L_{1}\xrightarrow{\delta^{T}_{1}\circ(\bar \delta_{2}^f \otimes 1)} S
\end{equation}
%%%G
We then consider the Koszul complex $K(\gamma_{1},\ldots,\gamma_{s};S)$ and set $G_{\bullet}:=\Tot(L_{\bullet}\otimes_{S}K_{\bullet}(\gamma;S))$
%%%
$$G_{\bullet}:0\ra G_{r+s-1}\ra\ldots\ra G_{r}\ra G_{r-1}\ldots\ra G_{1}\ra G_{0}\ra 0$$
\begin{equation}\label{G}
G_i=\left\lbrace
           \begin{array}{c l}
              S         & i=0,\\
           \bigoplus_{j=1}^{i}S(-j) ^{n_{ij}}          & 1\leq i \leq r-1, \text{for~some~} n_{ij},\\
           \bigoplus_{j=i-r+2}^{i}S(-j) ^{n_{ij}}            & r \leq i \leq  r+s-1, \text{for~some~} n_{ij}.
           \end{array}
         \right.
         \end{equation}
        
    The following lemma is essential to verify the acyclicity of free approaches of residual intersections.    
%%%%%Lema1
\begin{Lemma}\label{I=a} Keeping the notation above and assume that $\mathcal{L}$ is the defining ideal of the symmetric algebra $\mathcal{S}_{I}$. Then 
\begin{itemize} 
\item[(i)]{ If $I=\fa$, then $(\gamma)+\mathcal{L}=(\bt)$.}
\item[(ii)]{Let $(R,\fm)$ be a local ring and   assume that $s\geq r$ and that  $I$ is minimally generated by $r$ elements. If $I=\fa$ then    $(\gamma_{1},\ldots,\gamma_{s})=(\bt)$. }
\end{itemize}
\end{Lemma}
\begin{proof} $(i).$ To describe the defining ideal of $\mathcal{S}_I$, consider a presentation matrix of $I$,
\begin{equation} \label{symmetric}
R^{m}\xrightarrow{\left(\begin{matrix}
b_{11}&\ldots&b_{1m}\\
\vdots&\vdots&\vdots\\
 b_{r1}&\ldots&b_{rm}\\
\end{matrix}\right)} R^{r}\xrightarrow{\phi} I \to 0
\end{equation}
Setting $
(l_{1}\ldots l_{m})=(t_{1}\ldots t_{r})
\left(\begin{matrix}
b_{11}&\ldots&b_{1m}\\
\vdots&\vdots&\vdots\\
 b_{r1}&\ldots&b_{rm}\\
\end{matrix}\right),$  one has $\mathcal{L}=(l_{1},\ldots,l_{m})S$.  The map $\phi$ induces the following exact sequence
$$0\longrightarrow \mathcal{L}_{[1]}\lra \sum_{i=1}^{r}Rt_{i}\xra{\varphi}I\to 0,\quad \varphi(t_{i})=f_{i}.$$
We show that $\varphi$ induces a well-defined map $\bar \varphi:\sum_{i=1}^{s} R\gamma_{i}\to \fa$ sending $\gamma_{i}$ to $a_{i}$. We have to show that for any $p_{ij}\in R$ such that $\sum_{i=1}^{s} p_{ij}\gamma_{i}=0$, we get $\sum_{i=1}^{s} p_{ij}a_{i}=0$. This amounts to say that for the syzygy matrix $(p_{ij})$ of the matrix $(c_{ij})$, one  has  $(a_{1}\ldots a_{s})(p_{ij})=0$.  Since $(a_{1}\ldots a_{s})=(f_{1}\ldots,f_{r})(c_{ij})$, hence $(a_{1}\ldots a_{s})(p_{ij})=((f_{1}\ldots,f_{r})(c_{ij}))(p_{ij})=(f_{1}\ldots,f_{r})((c_{ij})(p_{ij}))=(f_{1}\ldots,f_{r})0=0$.
We then have the following diagram
\begin{equation}\label{snake}
 \xymatrix{
         0\ar[r]&      \mathcal{L}\ar[r]            &  \sum_{i=1}^{r}Rt_{i}\ar^{\varphi}[rr]      &&I\ar[r]             & 0                       \\
        0\ar[r] & \mathcal{L}\cap \sum_{i=1}^{s} R\gamma_{i}\ar@{-->}[u]\ar[r] &\sum_{i=1}^{s} R\gamma_{i}\ar[u]\ar[rr]      &&        \fa\ar[r]\ar@{^{(}->}[u] &0   \\
        }
\end{equation}
Now, having $I=\fa$, the snake lemma implies that $\sum Rt_{i}=\mathcal{L}+\sum R\gamma_{i}$ as $R$-modules. Therefore, we have the equality $(\bt)=\mathcal{L}+(\gamma)$ as $S$-ideals.

$(ii).$ Consider the system of equations
\begin{equation}
\left\lbrace
           \begin{array}{cl}
              \gamma_{1}=   & c_{11}t_{1}+\ldots+c_{r1}t_{r},\\
           \vdots          &\quad\quad\vdots\\
           \gamma_{s}=          & c_{1s}t_{1}+\ldots+c_{rs}t_{r}.
           \end{array}
         \right.
         \end{equation}

Since $\fa\not\subseteq \fm I$, one of the coefficients $c_{ij}$ is invertible, say $c_{11}$. Applying the elimination on the above system of equations, we have $(\gamma_{1},\ldots,\gamma_{s})=(t_{1},\gamma'_{2},\ldots,\gamma'_{s})$. Similarly, $I=(a_{1},a_{2}-c_{12}a_{1},\ldots,a_{s}-c_{1s}a_{1})$. Unless $\mu(I)=1$, there exists $2\leq i\leq s$ such that $a_{i}-c_{1i}a_{1}\not\in \fm I$, say $i=2$. Hence one of the coefficients, $c_{j2}-c_{12}c_{j1}$ is invertible. Then we eliminate $t_{2}$ from the expression of $(\gamma'_{3},\ldots,\gamma'_{s})$ to get $(\gamma''_{3},\ldots,\gamma''_{s})$ such that $(\gamma_{1},\ldots,\gamma_{s})=(t_{1},t_{2},\gamma''_{3},\ldots,\gamma''_{s})$.  This procedure stops at the $r=\mu(I)$'th step. In that case 
$(\gamma_{1},\ldots,\gamma_{s})=(t_{1},\ldots,t_{r},\gamma'''_{r+1},\ldots,\gamma'''_{s})=(\bt)$.
\end{proof}
%%%%%Example
The following example shows that the conditions $s\geq r$ and $\mu(I)=r$ in the \autoref{I=a}(ii) are necessary.
\begin{Example} Let $R=k[x,y,z]$, $I=(xy+xz,xz-x,xy)$ and $\fa=(x)$. Here, $\fa=I$ and $a_{1}=f_{1}-f_{2}-f_{3}$. Hence, $\gamma_{1}=t_{1}-t_{2}-t_{3}$, clearly, $(\gamma_{1})\neq (t_{1},t_{2},t_{3})$. 
Now, let $\fa=(xy+x^{2}z,x^{2}z-x,xy)$ and $I=(x,xz)$. $\fa=I=(x)$. However, setting $\gamma_{1}=xt_{2}+yt_{1},\gamma_{2}=xt_{2}-t_{1},\gamma_{3}=yt_{1}$, we have $(\gamma_{1},\gamma_{2},\gamma_{3})=(t_{1},xt_{2})\neq (t_{1},t_{2})$.
\end{Example}
Since $G_{\bullet}$ is a homogenous complex of $S$-modules, $H_{i}(G_{\bullet})$ are graded $S$-modules, so that the local cohomology modules $H^{j}_{\bt}(H_{i}(\G_{\bullet}))$ are graded modules. We denote the degree $i$ piece of this module by $H^{j}_{\bt}(H_{i}(\G_{\bullet}))_{[i]}$.
%%%%%Lema 2
\begin{Lemma}\label{lem}
 Suppose that any of the following cases holds.
\begin{itemize}
\item[(i)]{$R$ is a local ring, $I=\fa$, $s\geq r$ and $\mu(I)=r$;}
\item[(ii)]{$I=\fa=(1)$.}
\end{itemize}
Then  $H^{j}_{\bt}(H_{i}(\G_{\bullet}))_{[0]}=0$ for all $j\geq 0$ and $i>0$; furthermore  $H_{0}(G_{\bullet})_{[0]}=R$. 
\end{Lemma}
\begin{proof} $(i)$. Since $G_{\bullet}$ is the totalization of $L_{\bullet}\otimes_{S}K_{\bullet}(\gamma;S)$. The convergens of spectral sequences shows that $H_{i}(G_{\bullet})$, for all $i$, has a filtration of subquotients of $H_{j}(\gamma;S)$. Hence,  $H_{i}(G_{\bullet})$  are $(\gamma)$-torsion for all $i$. Now, according to  \autoref{I=a},  $(\bt)=(\gamma)$, so that $H^{j}_{\bt}(H_{i}(G_{\bullet}))=0$ for all $i$ and for  all $j>0$. For $j=0$, we have $H^{0}_{\bt}(H_{i}(G_{\bullet}))=H_{i}(G_{\bullet})$ which is a subquotient of $G_{i}$. Regarding \autoref{G}, we get $(G_{i})_{[0]}=0$ for all $i>0$ and $H_{i}(G_{\bullet})_{[0]}=(G_{0})_{[0]}=R$ . This completes the proof of part $(i)$.

$(ii)$. Since $I=(1)$, the Koszul complex $K_{\bullet}(f_{1},\ldots, f_{r};R)$ is split exact. In particular the maps $\bar \delta_{i}^f \otimes 1$ in \autoref{black1} are all surjective. 
We then show that  $L_{\bullet}$ and the approximation complex $\mathcal{Z}_{\bullet}$ are quasi-isomorphisms.
\begin{equation}\label{ZL0}
 \xymatrix{
  L_{\bullet}:               0\ar[r]& L_{r-1}\ar^{\bar\delta_{r}^f \otimes 1}[d]\ar[r]      &L_{r-2}\ar[r]\ar^{\bar\delta_{r-1}^f \otimes 1}[d]      &\ldots\ar^{( \delta_{3}^f \otimes 1)\oplus \delta^{T}_{3}}[r]&L_{1}\ar[r]^{\delta^{T}_{1}\circ(\bar \delta_{2}^f \otimes 1)}\ar^{\bar\delta_{2}^f \otimes 1}[d]& S\ar[d]  \\
 \mathcal{Z}_{\bullet}:0\ar[r]& Z_{r-1}\bigotimes_{R}S(-r+1)\ar^{ \delta^{T}_{r-1}}[r]& Z_{r-2}\otimes S(-r+2) \ar^(.7){ \delta^{T}_{r-2}}[r] & \ldots\ar^(0.3){\delta^{T}_{2}}[r] & Z_1\bigotimes_{R}S(-1)\ar^{\delta^{T}_{1}}[r]&  Z_{0}\bigotimes_{R}S\\ 
}
\end{equation}
The fact that $\overline{ \delta_{i+1}^f \otimes 1}: H_{i}(L_{\bullet})\to H_{i}(\mathcal{Z}_{\bullet})$  is an isomorphism for $i\geq 2$, is a straightforward conclusion from the convergence of the spectral sequences driven from double complex \autoref{black1}.

We show that $\overline {\delta_{2}^f \otimes 1}: H_{1}(L_{\bullet})\longrightarrow   H_{1}(\mathcal{Z}_{\bullet})$ is an isomorphism. 
 Cosider the map $$\overline{\delta_{2}^f \otimes 1}:\frac{\ker(\delta^{T}_{1}\circ(\bar \delta_{2}^f \otimes 1))}{\im(\delta_{3}^f)\oplus \im(\delta^{T}_{3})}\longrightarrow  \frac{\ker(\delta^{T}_{1})}{\im(\delta^{T}_{2})}$$
Surjectivity: set $d_{2}=\bar\delta_{2}^f \otimes 1$ and $d_{3}=\bar\delta_{3}^f \otimes 1$. Let $z_{1}\in \ker(\delta^{T}_{1})$. Since $d_{2}:L_{1}\to \mathcal{Z}_{1}$ is surjective there exists $x_{2}\in L_{1}$ with $d_{2}(x_{2})=z_{1}$. Then $\delta^{T}_{1}\circ d_{2}(x_{2})=\delta^{T}_{1}(z_{1})=0$. Thus $x_{2}\in \ker(\delta^{T}_{1}\circ(\bar \delta_{2}^f \otimes 1))$ and $\bar d_{2}(\bar x_{2})=\bar z_{1}$.

Injectivity: Let $\bar x_{2}\in ker(\bar d_{2})$. Then $d_{2}(x_{2})=\delta^{T}_{2}(z_{2})$ for some $z_{2}\in \mathcal{Z}_{2}$. Since $d_{3}:L_{2}\to \mathcal{Z}_{2}$ is onto, there exists $x_{3}\in L_{2}$ such that $d_{3}(x_{3})=z_{2}$. By commutativity of the diagram \autoref{black1}, we have $d_{2}(\delta_{3}^{T}(x_{3}))=\delta_{2}^{T}(d_{3}(x_{3}))=\delta_{2}^{T}(z_{2}).$ Hence, $d_{2}(\delta_{3}^{T}(x_{3}))=d_{2}(x_{2})$. Therefore, $x_{2}-\delta_{3}^{T}(x_{3})\in ker(d_{2})=\im(\delta_{3}^f \otimes 1)$. Thus, $x_{2}\in \im(\delta_{3}^f)\oplus \im(\delta^{T}_{3})$.

We return to the complex $G_{\bullet}=\Tot(L_{\bullet}\otimes_{S}K_{\bullet}(\gamma;S))$. Since the homologies of this complex are filtered by sub-quotients of homologies of the Koszul complex $K_{\bullet}(\gamma; S)$, $H_{i}(G_{\bullet})$ are $(\gamma)$-torsion. On the other hand,  $H_{i}(G_{\bullet})$ is filtered by sub-quotients of  homologies of $\mathcal{Z}_{\bullet}$, since  $\mathcal{Z}_{\bullet}\simeq L_{\bullet}$. Since $H_{i}(\mathcal{Z}_{\bullet})$ are $S/\mathcal{L}$-modules, it follows that $H_{i}(G_{\bullet})$ are $\mathcal{L}$-torsion and hence $((\gamma)+\mathcal{L})$-torsion. \autoref{I=a}(i), then implies that  $H_{i}(G_{\bullet})$ are $(\bt)$-torsion. Accordingly, $H^{j}_{\bt}(H_{i}(G_{\bullet}))=0$ for all $i$ and for  all $j>0$. As for $j=0$, we have $H^{0}_{\bt}(H_{i}(G_{\bullet}))=H_{i}(G_{\bullet})$ which is a subquotient of $G_{i}$.  The argument then completes as that of part $(i)$.
\end{proof}
Consider the complex $G_{\bullet}$. 
Since $G_{i}=\oplus S(-j)$, $H^{j}_{\bt}(G_{i})=0$ for $j<r$. For $j=r$ and $i\leq r-1$, according to  the formulas in \autoref{G}, we have $H^{r}_{\bt}(G_{i})=\bigoplus^{n_{ij}}\bigoplus_{j=1}^{i}H^{r}_{\bt}(S)(-j).$  Since $\End(H^{r}_{\bt}(S))=-r$, we get $H^{r}_{\bt}(G_{i})_{[0]}=0$ for all $i\leq r-1$. Hence 
 $$H^{r}_{\bt}(G_{\bullet})_{[0]}=0\lra H^{r}_{\bt}(G_{r+s-1})_{[0]}\lra \ldots\lra  H^{r}_{\bt}(G_{r})_{[0]}\lra 0.$$
  For each $i$, $H^{r}_{\bt}(G_{i})_{[0]}$ is a finite free $R$-module. We call the terms of $H^{r}_{\bt}(G_{\bullet})_{[0]}$ by $F_{i}$ that is 
 \begin{equation}\label{ZL}
 \xymatrix{
              0\ar[r]& H^{r}_{\bt}(G_{r+s-1})_{[0]}\ar[r]\ar@{=}[d]      &H^{r}_{\bt}(G_{r+s-2})_{[0]}\ar[r]\ar@{=}[d] &\ldots\ar[r]&H^{r}_{\bt}(G_{r})_{[0]}\ar@{=}[d] \\
 0\ar[r]& F_{s}\ar[r]& F_{s-1}\ar[r] & \ldots\ar^{d_{2}}[r] & F_{1}.\\ 
}
\end{equation}
%%%%%%%%%%%%%%%%%%%%%prop1

The next proposition introduces the free complex, which will be the free approach of residual intersections.
\begin{Proposition}\label{F} Let  $R$ be  a Noetherian ring of dimension
$d$ and $I$ contains a regular element. Keeping the   notation of \autoref{ZL}, there exists a map $d_{1}:F_{1}\lra R$   with $\tau:=\im(d_{1})$ such that $d_{1}\circ d_{2}=0$ and  
\begin{itemize}
\item[(i)]{$\tau\subseteq J$.}
\item[(ii)]{$\tau_{\fp}=J_{\fp}=\fa_{\fp}$ for all $\fp\in \Spec(R)\setminus \V(I)$.}
\item[(iii)]{$\Supp(H_{i}(F_{\bullet}))\subseteq \V(\fa)$ where $F_{\bullet}:0\to F_{s}\to\ldots F_{2}\xrightarrow{d_{2}} F_{1} \xrightarrow{d_{1}} R\to 0$.}
\item[(iv)]{${\sqrt\fa}\subseteq {\sqrt \tau}$} (indeed, $\fa\subseteq \tau$ after \autoref{taustr}).
%\item[(v)]{If $f_{1},\ldots,f_{r}$}
\end{itemize}
\end{Proposition}
\begin{proof} 
 In order to get the zero term of the complex $F_{\bullet}$, we tensor $G_{\bullet}$ with the \v{C}ech complex of $S$ with respect to the sequence $\bt$, $C^{\bullet}_{\bt}:0\ra C^{r}\ra\ldots\ra C^{1}\ra 0$. We then study the horizontal and vertical spectral sequences derived from the double complex $C^{\bullet}_{\bt}\otimes_{S}G_{\bullet}$.
Setting this double complex in the third quadrant of the cartesian plane, we get 
\begin{equation}
(^{1}\E_{ver}^{-i,-j})_{[0]}=\left\lbrace
           \begin{array}{cl}
              0  & j\neq r,\\
             H^{r}_{\bt}(G_{i})_{[0]}&j=r, i\geq r,\\
           0     & j=r, i<r
           \end{array}
         \right.
         \quad \text {and~}\quad ^{1}\E_{hor}^{-i,-j}=C^{j}_{\bt}\otimes_{S}H_{i}(G_{\bullet}).
         \end{equation}

The second terms of the horizontal spectral  sequence are  $^{2}\E_{hor}^{-i,-j}=H^{j}_{\bt}(H_i(\G_{\bullet}))$. 
 The following diagram depicts these spectral sequences wherein $H_i:=H_i(\G_{\bullet})$. The line at the bottom is the only non-zero term of $^{1}\E_{ver}^{-i,-j}$. The blue arrow is the map induced by the convergence of the spectral sequence.
\begin{equation}\label{bigdouble}
 \xymatrix{
 H^{0}_{\bt}(H_{s+r-1})  &          &               \ldots   &H^{0}_{\bt}(H_2)&  H^{0}_{\bt}(H_1) \ar@[--->]_(.3){ ^{2}d^{-1,0}_{hor}}[ldd]  &H^{0}_{\bt}(H_0) \\
 H^{1}_{\bt}(H_{s+r-1})  &          &               \ldots    &H^{1}_{\bt}(H_2)&  H^{1}_{\bt}(H_1)      & H^{1}_{\bt}(H_0)                        \\
 H^{2}_{\bt}(H_{s+r-1})  &          &               \ldots    &H^{2}_{\bt}(H_2)&  H^{2}_{\bt}(H_1)      &\vdots                                         \\
 H^{3}_{\bt}(H_{s+r-1})  &          &               \ldots    &H^{3}_{\bt}(H_2)&   \vdots                    &\vdots                                         \\
              \vdots            &          &             \vdots      &          \vdots     &   \vdots                    & \vdots                                        \\
H^{r}_{\bt}(H_{s+r-1})   &  \ldots  & H^{r}_{\bt}(H_r)   &          \ldots     &    \ldots                   & \ldots.                                        \\
 \textcolor{blue}{H^{r}_{\bt}(G_{r+s-1})}\ar[r]&\ldots\ar^{\phi}[r]& \textcolor{blue}{H^{r}_{\bt}(G_{r})}\ar[r]\ar@[blue][uuuuuurrr]& \textcolor{blue}{H^{r}_{\bt}(G_{r-1})}\ar[r]&\ldots\ar[r]& \textcolor{blue}{H^{r}_{\bt}(G_{0})}\\
  }
 \end{equation}
Since $H^{r}_{\bt}(G_{r-1})_{[0]}=\ldots=H^{r}_{\bt}(G_{0})_{[0]}=0$, we get $\coker(\phi_{[0]})=(^{\infty}\E_{ver}^{-r,-r})_{[0]}$ which is isomorphic to the homology module of the total complex. By the convergence of the spectral sequence, there exists 
\begin{equation}\label{F1}
\mathcal{F}_{1}\subseteq  \coker(\phi_{[0]}){\text ~~ such~that~~}\coker(\phi_{[0]})/\mathcal{F}_{1}\simeq (^{\infty}\E_{hor}^{0,0})_{[0]}.
\end{equation} 

Notice that $(^{\infty}\E_{hor}^{0,0})_{[0]}\subseteq H^{0}_{\bt}(H_0(G_{\bullet}))_{[0]}\subseteq  H_0(G_{\bullet})_{[0]}$. 
 
According to the structure of $G_{\bullet}$ in \autoref{G}, 
\begin{equation}\label{H0G}
H_{0}(G_{\bullet})=\frac{S}{\im(\delta^{T}_{1}\circ(\bar \delta_{2}^f \otimes 1))+\im(\delta_{1}^{\gamma}\otimes 1)}=\frac{S}{(f_{i}t_{j}-f_{j}t_{i},\gamma_{1},\ldots,\gamma_{s})}.
\end{equation}
In particular, $H_{0}(G_{\bullet})_{[0]}=R$. We define  $d_{1}:=\psi_{[0]}$ where $\psi$ is given by the following chain of compositions.
\begin{equation}\label{psimap}
\psi := H^{r}_{\bt}(G_{r})\xrightarrow{\pi} \coker(\phi)\to \coker(\phi)/\mathcal{F}_{1}\simeq (^{\infty}\E_{hor}^{0,0}) \subseteq H^{0}_{\bt}(H_0(G_{\bullet}))\subseteq  H_0(G_{\bullet}).
\end{equation} 
 Since $\pi\circ d_{2}=0$, we have $d_{1}\circ d_{2}=0$.

Now, we prove properties  $(i)-(iv)$ of the complex $F_{\bullet}$.  Set $\mathcal{T}=\im(\psi)$. Then $\tau=\mathcal{T}_{[0]}$.  Since $H^{r}_{\bt}(G_{r})_{[j]}=0$ for all $j\geq 1$,  we have $(^{\infty}\E_{hor}^{0,0})_{[j]}=0$ for all $j\geq 1$. That is $\mathcal{T}_{[j]}=0$ for $j\geq 1$. \footnote {There are some bonus here,  since  $H^{r}_{\bt}(G_{r-i})_{[j]}=0$ for all $i,j\geq 1$, we have  $(^{\infty}\E_{hor}^{-p,-q})_{[j]}=0$ for all $p<q$ and all $j\geq 1$. Additionally, because  $H_{0}(G_{\bullet})$ is concentrated in positive degrees, we have  $\tau=\mathcal{T}_{[0]}=\mathcal{T}$.}  One then has, $\bt \tau\subseteq \mathcal{T}_{[1]}=0$.

Regarding \autoref{H0G}, we get 
\begin{align*}
\tau &\subseteq (f_{i}t_{j}-f_{j}t_{i},\gamma) :_{R} (\bt) \\
     &\subseteq (f_{i}t_{j}-f_{j}t_{i},\gamma,\mathcal{L}) :_{R} (\bt) 
     = (\gamma,\mathcal{L}) :_{R} (\bt) \\
     &= \frac{(\gamma,\mathcal{L})}{\mathcal{L}} :_{R} \frac{(\bt)}{\mathcal{L}}\subseteq  \frac{(\gamma,\mathcal{L})_{[1]}}{\mathcal{L}_{[1]}} :_{R} \frac{(\bt)_{[1]}}{\mathcal{L}_{[1]}}
     = \fa :_{R} I = J.
\end{align*}

Recall that $\mathcal{L}$ is the defining ideal of the symmetric algebra. It is clear that for all $i$ and $j$, $f_{i}t_{j}-f_{j}t_{i}\in  \mathcal{L}$. The equality in the last lines follows from the same argument as in \autoref{snake}. Hence, the proof of $(i)$ follows. 

$(ii)$. Localizing at $\fp \in \Spec(R)\setminus \V(I)$, we assume that $I=(1)$.  As we see in the proof of \autoref{lem} at \autoref{ZL0}, in the case $I=(1)$,  $L_{\bullet}$ and $\mathcal{Z}_{\bullet}$ are quasi-isomorphic.   For all $i$, $H_{i}(\mathcal{Z}_{\bullet})$ is $\mathcal{L}$-torsion, hence $H_{i}(L_{\bullet})$ and  thus $H_{i}(G_{\bullet})$ is $\mathcal{L}$-torsion. Since $H_{i}(G_{\bullet})$ is as well $(\gamma)$-torsion, these homology modules are $((\gamma)+\mathcal{L})$-torsion. Notice that $(\bt)_{[1]}/(\mathcal{L}+(\gamma))_{[1]}\simeq I/\fa$. The latter is a cyclic module since $I=(1)$. That is, there exists $t_{0}\in (\bt)$ such that $(\bt)=(\mathcal{L}+(\gamma))+St_{0}$. It then follows that $H^{j}_{\bt}(H_{i}(G_{\bullet}))=H^{j}_{t_{0}}(H_{i}(G_{\bullet}))=0$ for all $j\geq 2$. Consequently, $(\E_{hor})_{[0]}$ collapses at the second page and $(^{\infty}\E_{hor}^{0,0})_{[0]}=H^{0}_{\bt}(H_{0}(G_{\bullet}))_{[0]}$.

Since $R$ is local and $I=(1)$ some $f_{i}$ are invertible, say $f_{1} $ is invertible.  In the definig ideal of $H_{0}(G_{\bullet})$, we can replace $t_{j}$ with $f_{j}f_{1}^{-1}t_{1}$, hence $\gamma_{i}=\sum c_{ji}t_{j}=\sum c_{ji}f_{j}f_{1}^{-1}t_{1}=a_{i}f_{1}^{-1}t_{1}$. Therefore, we get an isomorphism 
$$\frac{S}{(f_{i}t_{j}-f_{j}t_{i},\gamma_{1},\ldots,\gamma_{s})}\simeq \frac{R[t_{1}]}{(a_{1}t_{1},\ldots,a_{s}t_{1})}.$$

Since we supposed $f_{1}$ to be invertible, $t_{0}$ in the previous paragraph can be $t_{1}$. Therefore, we have 
$$(^{\infty}\E_{hor}^{0,0})_{[0]}=H_{t_{1}}^{0}(R[t_{1}]/(a_{1}t_{1},\ldots,a_{s}t_{1}))_{[0]}=(a_{1},\ldots,a_{s}).$$
Since $\tau=(^{\infty}\E_{hor}^{0,0})_{[0]}$ the proof is complete.

$(iii).$ Let $\fp\in \Spec(R)\setminus \V(\fa)$. Then $\fa_{\fp}=I_{\fp}=(1)$. \autoref{lem}(ii) then implies that $H^{j}_{\bt}(H_{i}(\G_{\bullet}))_{[0]}=0$ for all $j\geq 0$ and $i>0$, and   $H_{0}(G_{\bullet})_{[0]}=R$. Moreover, since $(^{2}\E_{hor}^{-i,-j})_{[0]}=H^{j}_{\bt}(H_{i}(\G_{\bullet}))_{[0]}=0$ for $i>0$, we have $\tau=(^{\infty}\E_{hor}^{0,0})_{[0]}=H_{0}(G_{\bullet})_{[0]}=R$. Hence $H_{0}(F_{\bullet})=0$.  

Since $(^{2}\E_{hor}^{-i,-j})_{[0]}=0$ for all $j>0$, the filtration generated for $\mathcal{F}_{1}$ by means of $(^{\infty}\E_{hor}^{-i,-j})_{[0]}$ is the zero filtrarion. Therefore $\mathcal{F}_{1}=0$. According to \autoref{psimap}, we then have $\ker(d_{1})=\ker(\psi_{[0]})=\ker (\pi)=\im(d_{2})$. This implies that $H_{1}(F_{\bullet})=0$.

 That $H_{i}(F_{\bullet})=0$ for all $i>1$,  follows from the convergence of spectral sequences. 

$(iv)$. This part is a consequence of $(iii)$ since $H_{0}(F_{\bullet})=R/\tau$.
\end{proof}
%%%%Prop 2
By \autoref{F},  it is clear that ${\sqrt \fa}\subseteq {\sqrt \tau}\subseteq {\sqrt J}$. Apparently, $\tau$ is closer to $\fa$ than $J$. However,  we see in the following proposition that for special ideals $I$, it goes more to the right. 

\begin{Proposition}\label{radical} Let $R$ be a Noetherian ring of dimension
$d$ and $I$ contains a regular element. Consider the  complex $F_{\bullet}$ defined in \autoref{F} and $\tau=\im(d_{1})$. Let $\Ht(J)=\alpha$. Suppose that $s\geq r$ and  $I$ is $r$-minimally generated from height $\alpha-1$. Then ${\sqrt \tau}={\sqrt J}$.
\end{Proposition}
\begin{proof} We first show that $\Ht(\tau)\geq \alpha$.  Let $\fp$ be a prime ideal of height $\alpha-1$. Then $ J \not\subseteq \fp$ because  $\Ht(J)=\alpha$ by assumption. 

Since $ J \not\subseteq \fp$, we have $\fa_{\fp}=I_{\fp}$. There are two cases, if $\fp\not\supseteq I$, then $\fa_{\fp}=I_{\fp}=(1)$. According to \autoref{lem}(ii), the convergence of spectral sequences in the proof of \autoref{F} implies that $F_{\bullet}$ is exact. In particular, $H_{0}(F_{\bullet})_{\fp}=0$, that is $\fp\not\supseteq \tau$.

 If $\fp\supseteq I$, then $\mu(I_{\fp})=r$, $s\geq r$ and $I_{\fp}=\fa_{\fp}$. Then  \autoref{lem}(i) infers that $F_{\bullet}$ is exact, thence $\fp\not\supseteq \tau$. Therefore  any  prime ideal of height $\alpha-1$ does not contain $\tau$, hence $\Ht(\tau)\geq \alpha$.

Now, let $\fq$ be a prime ideal that contains  $\tau$. There are two cases, if  $\fq\not\supseteq I$, then \autoref{F}(ii), implies that  $\tau_{\fq}=J_{\fq}$ that is  $\fq\supseteq J$. 

If $\fq\supseteq I$ then  $\mu(I_{\fq})=r$, since $I$ is $r$-minimally generated from height $\alpha-1$ and $\Ht(\fq)\geq \alpha$ . If by contrary, $\fq\not\supseteq J$ then $I_{\fq}=\fa_{\fq}$.  Then again \autoref{lem}(i) implies that $\tau_{\fq}=R_{\fq}$ which is a contradiction. Hence in each case $\fq\supseteq \tau $ implies that $\fq\supseteq J$,  equivalently ${\sqrt \tau}\supseteq {\sqrt J}$. On the other hand,  one always has ${\sqrt \tau}\subseteq {\sqrt J}$ which completes the proof. 
\end{proof}

We are now, ready to state the main theorem of this section. That, residual intersections of $r$-minimally generated ideals admit  free approaches.
%%%%%%%%THeorem 1
\begin{Theorem}\label{residual1}
Let $R$ be a Noetherian ring, $\fa=(a_{1},\ldots,a_{s})$ a sub-ideal of $I$ and $J=\fa:I$. Assume $I$ contains a regular element and is $r$-minimally generated from height $s-1$. If $\Ht(J)\geq s$ , i.e. $J$ an algebraic $s$-residual intersection of $I$, and $s\geq r$  then
\begin{itemize}
\item[(i)]{ $\Ht(J)=s$;}
\item[(ii)]{Let $t$ be an integer and assume that $\depth(R_{\fp})\geq s$ for all $\fp$ with $s\leq \Ht(\fp)\leq t$. Then for any  $\fp\in \min(J)$, either $\Ht(\fp)=s$ or $\Ht(\fp)\geq t+2$;}
\item[(iii)]{If $R$ satisfies the Serre's condition S$_{s}$ then $J$ admits a free approach.}
\end{itemize}
\end{Theorem}
\begin{proof}$(i)$.  Let $\Ht(J)=s+i$. For any prime $\fp$ with $s-1\leq \Ht(\fp)\leq (s+i-1)$, we have $I_{\fp}=\fa_{\fp}$. By a similar argument to that in the proof of \autoref{radical}, we see that $(F_{\bullet})_{\fp}$ is exact for all such $\fp$. Now, let $\fq\supseteq J$ and $\Ht(\fq)=s+i$. Since ${\sqrt \tau}={\sqrt J}$ by \autoref{radical}, $H_{0}(F_{\bullet})_{\fq}=(R/\tau)_{\fq}$ is non-zero and of  finite length. On the other hand, $(F_{\bullet})_{\fq}$ is exact on the punctured spectrum. Therefore $H_{i}(F_{\bullet})_{\fq}$ are all of finite length. The new intersection theorem then implies that the length of the complex $(F_{\bullet})_{\fq}$ is at least the dimension of the ring. That is $s\geq \dim(R_{\fq})=s+i$. Hence $i=0$.

$(ii)$. Let  $\fp$ be a prime ideal with  $\Ht(\fp)=s$. Then \autoref{lem}(i) implies that $(F_{\bullet})_{\fp}$ is exact on the punctured spectrum. Since the length of the complex is $s$ and the depth of each component is at least $s$, Peskine-Szpiro's acyclicity Lemma \cite[1.4.24]{bruns1998cohen} implies that $(F_{\bullet})_{\fp}$  is acyclic.  We then inductively, show that $(F_{\bullet})_{\fp}$ is acyclic  for any prime ideal $\fp$ with $\Ht(\fp)\leq t$.  Hence the new intersection theorem implies that  $(R/\tau)_{\fp}$ is not of finite length. Hence  any prime ideal $\fp$ with $s<\Ht(\fp)\leq t$ does not belong to $\min(\tau)=\min(J)$. Now let $\fq$ be a prime ideal of height $t+1$. Since $(F_{\bullet})_{\fq}$ is acyclic on the punctured spectrum and $ \dim(R_{\fq})>s$,  once more, the new intersection theorem implies that $\fq\not\in \min(J)$.  That completes the proof of part $(ii)$.

$(iii)$. The proof of this part is the same as part $(ii)$. We use \autoref{lem}  to see that $(F_{\bullet})_{\fp}$ is exact  at height $s-1$ or less. Then we use the acyclicity Lemma inductively, appealing to the condition of the depth of the ring. Thus $F_{\bullet}$ is an acyclic complex of finite free modules of length $s=\Ht(J)$, $F_{0}=R$ and $H_{0}(F_{\bullet})=R/\tau$ with $\tau\subseteq J$ and ${\sqrt \tau}={\sqrt J}$, according to \autoref{F}(i) and \autoref{radical}. 
\end{proof}
Although the next result, \autoref{geometric},  applies for complete intersections, we wonder if the following question has an affirmative answer.
\begin{Question} In a ring $R$, let $I$ be an  $r$-minimally generated ideal that is not a complete intersection. Does $I$  admit any geometric $s$-residual intersections for some $s\geq r$?
\end{Question}
For instance, let $R=\mathbb{C}[x_{0},\ldots,x_{5}]/(x_{0}^{2},x_{1}^{2})$, then $I=(x_{0},x_{1},x_{2},x_{3})$ is $4$-minimally generated ideal from height $2$. Observe  $I$ does not admit any $3$-residual intersections, since $\mu(I_{\fp})=4$ for all $\fp\supseteq I$. We still do not know if there exist $a_{1},\ldots,a_{4}\in I$ such that $\Ht((a_{1},\ldots,a_{4}):I)\geq 4$.
\begin{Corollary}\label{geometric} Let $R$ be a Noetherian ring satisfying the Serre's condition S$_{s+1}$, $\fa=(a_{1},\ldots, a_{s})$ a sub-ideal of $I$ and $J=\fa:I$. Assume $I$ contains a regular element and is $r$-minimally generated from height $s-1$. If  $J$ is a geometric  $s$-residual intersection of $I$, and $s\geq r$  then $J$ is an unmixed ideal with the finite free resolution given by complex $F_{\bullet}$. 
\end{Corollary} 
\begin{proof}  By \autoref{residual1}(iii), the complex $F_{\bullet}$ is acyclic and resolves $R/\tau$. For any prime ideal $\fp$ with $\Ht(\fp)\geq s+1$, we have $\depth(R_{\fp})\geq s+1$. The acyclicity of the complex $(F_{\bullet})_{\fp}$ shows that $\depth(R/\tau)_{\fp}>0$. Hence $\fp\not\in \Ass(R/\tau)$. In particular, $\tau$ is an unmixed ideal of height $s$. We proceed to show that $\tau=J$ in the case where $J$ is a geometric $s$-residual intersection of $I$.  Let $\fp\in \Ass(\tau)$, then $\Ht(\fp)\leq s$, moreover  $\fp\supseteq J$ because  $\sqrt{\tau}=\sqrt{J}$ by \autoref{radical}. Since $J$ is a geometric $s$-residual intersection of $I$, $\fp\not \in\V(I)$. Thus \autoref{F}(ii) implies that $\tau_{\fp}=J_{\fp}$. Then $\tau=J$. The result follows from the fact that $\Ht(\tau)=\Ht(J)=s$, for any $\fp\in \Ass(\tau)$,  $\Ht(\fp)= s$, and that $\tau$ is resolved by $F_{\bullet}$ by \autoref{residual1}(iii).
\end{proof}
\begin{Remark} An important advantage of \autoref{geometric}, compared to similar results on residual intersections of complete intersections \cite{UlrichArtin, HamidNaeliton, boucca2019residual, CHU}, is that in this context, the base ring \( R \) is only required to satisfy \( \text{S}_{s+1} \) and not necessarily be Cohen-Macaulay.

\end{Remark}
%%%%%%%%%%%%%%%%%%%%%%%%%%%%%%%%%%%%%%
\subsubsection{multiplicity  of residual intersections}\label{subsec}

Since residual intersections serve as an algebraic tool in enumerative questions, finding an explicit formula for the Hilbert polynomial of \( R/J \), or at least an upper bound for the multiplicity, as a function of \( d \), \( s \), \( r \), \( \deg f_i \)'s,  \( \deg t_i \)'s and the Hilbert polynomial of $R$ is a significant goal in the theory of residual intersections.

The works of Chardin-Eisenbud and Ulrich \cite{CHU} and \cite{ceu} established that such bounds must exist. Further extensions were made in \cite{boucca2019residual} and \cite{HamidNaeliton}. 

So far, the complete intersection is the only class of ideals \( I \) for which the Hilbert function of an \( s \)-residual intersection \( J \) is computable in terms of the aforementioned parameters. Although this fact is known to experts, it is not well-documented in the literature. However, one can find the formula for the Hilbert function in Ph.D. thesis \cite{KevinThesis}.

Assume that
$ R=\bigoplus _{n \geq 0} R_{n}$ is a
 positively graded *local Noetherian ring of dimension $d$ over an Artinian local ring $
(R_{0}, \fm_{0})$ and  set $\fm=\fm_{0}+R_+$. Suppose that $I$ and
$\fa$ are homogeneous ideals of $R$ generated by homogeneous
elements $f_1,\ldots,f_r$ and $a_1,\ldots,a_s$, respectively.
 Let $\deg f_t=d_t$ for all $1 \leq t \leq r$ and $\deg a_{t}=l_{t}$ for $1\leq t \leq
s$. 
%%%%%%prop multiplicity
\begin{Proposition}\label{multiplicity} Keeping the above notation, assume that $f_{1},\ldots,f_{r}$ is a regular sequence and $R$ is a Cohen-Macaulay ring, and $J=\fa:I$ is an $s$-residual intersection. Then the Hilbert function of $R/J$ is a function of $d$, $s$, $r$, $d_i$'s, $l_i$'s  and the Hilbert function of $R$. We call the multiplicity of such a residual intersection {\bf ericci}. 
\end{Proposition}
\begin{proof} Under the mentioned condition,  \cite{boucca2019residual} shows that  $J=\Kitt(\fa,I)$ and the latter has a resolution by $\mathcal{Z}^{+}_{\bullet}$-complex. Since $I$ is a complete intersection,  the Hilbert function of each component of $\mathcal{Z}^{+}_{\bullet}$ is computable in terms of that of the tail of the Koszul complex of $I$. Indeed, the computation shows that  
$\HF_{R/J}(n)=\sum_{i=0}^{s}(-1)^{i}\HF_{F_{i}}(n)$ where $F_{\bullet}:0\to F_{s}\to\ldots F_{2}\xrightarrow{d_{2}} F_{1} \xrightarrow{d_{1}} R\to 0$ is the complex introduced in \autoref{F}. 
\end{proof}

%%%%prop

Interestingly, the bound {\bf ericci} is an upper bound for the multiplicity of a vast class of residual intersections even in non-Cohen-Macaulay rings.

\begin{Proposition}\label{ericci} Keep the graded setting, suppose that $R$ satisfies  Serre's condition S$_{s}$, $\fa=(a_{1},\ldots,a_{s})$ a sub-ideal of $I$ and $J=\fa:I$ an algebraic $s$-residual intersection of $I$. Assume $I$ contains a regular element and is $r$-minimally generated from height $s-1$. If $s\geq r$  then $e(R/J)\leq ericci$.

Furthermore, if $R$ satisfies Serre's condition S$_{s+1}$ and $e(R/J)= ericci$ then $J$ is an unmixed ideal with a finite free resolution given by the complex $F_{\bullet}$. 
\end{Proposition}
\begin{proof} According to \autoref{residual1}, the complex $F_{\bullet}$ provides a free resoution for $R/\tau$. Hence $\HF_{R/\tau}(n)=\sum_{i=0}^{s}(-1)^{i}\HF_{F_{i}}(n)$ which is the same as the Hilbert function of an $s$-residual intersection of a complete intersection $I$ given by degrees $d_{i}$'s. Therefore $e(R/\tau)=ericci$. On the other hand $\tau\subseteq J$ and $\V(\tau)=\V(J)$. Hence $e(R/J)\leq e(R/\tau)=ericci$. 

Now assume that $R$ satisfies Serre's condition S$_{s+1}$. As we see in the proof of  \autoref{geometric},   $\Ht(\fp)=s$ for all $\fp\in \Ass(R/\tau)$.  Therefore $\Ass(R/\tau)=\min(\tau)=\min(J)$.  Now we consider the exact sequence $0\lra\frac{J}{\tau}\lra \frac{R}{\tau}\lra \frac{R}{J}\lra 0.$ 
Since $\dim(R/J)=\dim(R/\tau)$ the equality of multiplicites implies that $\dim(J/\tau)<\dim(R/\tau)$. In particular, for any $\fp\in \Ass(R/\tau)$, $(J/\tau)_{\fp}=0$. Thence $\tau=J$. So that complex $F_{\bullet}$ provides a free resolution for $R/J$.
\end{proof}
\vspace{0.5cm}
Here is an illustrative example.
\vspace{0.5cm}
\begin{Example}\label{Examle4}
 Let $R=\mathbb{Q}[x_0,\ldots,x_5]/(x_0^2+x_1^2)$ and  $I=(x_0,x_1,x_2+x_3+x_4+x_5)$. Then $I$ is $3$-minimally generated from height $2$, moreover $\pdim(R/I)$ is infinite.  Let $a_{1},a_{2},a_{3}$ be three general elements of degree $2$ in $I$. Then $J=(\fa:I)$ is a $3$-residual intersection. Using Macaulay2, we check that  $J$ has infinite projective dimension, $\dim(R/J)=\depth(R/J)=2$ and  $e(R/J)=11$. Although $J$ has infinite projective dimension, according to \autoref{residual1}(iii), $J$ admits a free approach of length $3$.  Now   choose any regular sequence of linear forms in $R$ of length $3$, for instance, $I'=(x_{2},x_{3},x_{4})$ and an ideal generated by three general elements of degree $2$ in $I'$, say $\fa'$. Then $J'=(\fa':I')$ is a $3$-residual intersection with $\pdim(R/J')=3$ and  $e(R/J')=14$. In this example ericci=$14$. 
 \end{Example}

A similar phenomenon occurs with the Castelnuovo-Mumford regularity. Specifically, the regularity can be bounded above by the regularity of the residual intersections of complete intersections. However, this result holds only in cases where the ideal $\tau$ in \autoref{F} coincides with $J$.
\begin{Proposition}\label{regularity} Keeping the same notation as above,  \autoref{subsec}, one can construct the complex $F_{\bullet}$ associated to the set of generators of $\fa$ and $I$. Set $R/\tau=H_0(F_{\bullet})$. If $R$ is a Cohen-Macaulay $^*$local ring of dimension $d$, $\dim(R/\tau)=d-s$  and $F_{\bullet}$  is acyclic then $$\Reg(R/\tau)\leq \Reg(R)+\dim(R_0)+\sigma(\fa)-(s-r+1)\beg(I/\fa)-s,$$
where $\sigma(\fa)=\sum \deg(a_i)$. 
\end{Proposition}
\begin{proof}   Since  $R/\tau$ is Cohen-Macaulay, one needs to compute $\End(H^{d-s}_{\fm}(R/\tau))$. Appealing the complex $F_{\bullet}$ and Cohen-Macaulayness of $R$,  this amounts to compute $\End(\ker(H^{d}_{\fm}(F_s)\to H^{d}_{\fm}(F_{s-1}))$. The computation follows the approach outlined on pages 6387--6388 of \cite[Theorem 3.6]{hassanzadeh2012cohen}, with the modification that
$g$ in the cited reference should be replaced with $r$.
\end{proof}

%%%%%%%%%%%%%%%%%%%%%%%%%%%%%%%%%%%%%%
\subsection{residual intersections of ideals with sliding depth} \label{section3.2}
In this section, we focus on $s$-residual intersections $J=(a_{1},\ldots,a_{s}):I$ without considering the relation between $s$ and number of generators of $I$. The goal is to find a minimal approach for  $J$. In this generality, we impose more restrictions on the homological conditions of  $I$.  We keep the notation introduced in \autoref{ResidualSection}.  So $R$ is a Noetherian ring of dimension
$d$, $I=(f_1,\ldots,f_r)$ is an ideal of grade $g\geq1$,
$\fa=(a_1,\ldots,s_s)$ is an ideal contained in $I$, $
J=\fa:_{R}I$, and $S=R[t_1,\ldots,t_r]$ is a standard polynomial extension
of $R$ with indeterminates $t_{i}$'s. Let $K_{\bullet}=K_{\bullet}(f_{1},\ldots,f_{r};R)$ be the Koszul complex of $I$; and set $H_{i}$ and $Z_{i}$ to be the $i$th koszul homology and $i$th Koszul cycle of this complex.

We review some properties of approximation complex $\mathcal{Z}_{\bullet}$. 
 A sequence $f_1,\ldots,f_r$ in a Noetherian ring $R$, is called  a {\bf proper sequence} if $f_{i+1}H_j(f_1,\ldots,f_i)=0$ for all $j\geq 1$ and $i=0,\ldots,r-1$\footnote{ Permutations of proper sequences are not necessarily proper sequences.}. An almost complete intersection is generated by a proper sequence. If $R$ is a Noetherian local ring with infinite residue field then the ideal $I=(f_1,\ldots,f_r)$ is generated by a proper sequence if and only if $\mathcal{Z}_\bullet(f;R)$ is acyclic \cite[Theorem 12.9]{Approximation}. 
 %%%%Proposition
 \begin{Proposition}\label{propersequence} Let $R$ be a Noetherian local ring with infinite residue field and $I=(f_1,\ldots,f_r)$ be generated by a proper sequence. Then 
 \begin{enumerate}
 \item{ $\Reg_S(\mathcal{S}_I)=0$, the Castelnuovo-Mumford regularity of the symmetric algebra  is zero. }
 \item{ $\pdim_S(\mathcal{S}_I)\leq r-1$ if and only if $\pdim_R(Z_i(\ff))\leq r-i-1$ for all $i$.}
 \end{enumerate}
 \end{Proposition}
 \begin{proof} $(1)$. Let $C^{\bullet}_{\bt}(S)$ be the Cech complex of $S$ with respect to $(t_1,\ldots,t_r)$. Consider the double complex $C^{\bullet}_{\bt}(S)\otimes _S\mathcal{Z}_{\bullet}(\ff)$ and put it in the third quadrant of the cartesian plane with indices of  $\mathcal{Z}_i$ change on rows. Since $I=(f_1,\ldots,f_r)$ is  generated by a proper sequence the complex $\mathcal{Z}_{\bullet}(\ff)$ is acyclic. Moreover,  $C^{i}_{\bt}(S)\otimes _S-$ is an exact functor, hence the horizontal spectral sequence collapses on the second page with non-zero terms on the first column 
 \begin{equation}
^2\E_{hor}^{-i,-j}=\left\lbrace
           \begin{array}{cl}
              0  & i\neq 0,\\
           H^j_{\bt}(\mathcal{S}_I) &i=0.
           \end{array}
         \right.
         \end{equation}

 On the other hand $H^j_{\bt}(\mathcal{Z}_i)=0$ for $j<r$, since $\grade(\bt,\mathcal{Z}_i)=r$ for all $i$. Therefore the vertical spectral sequences, as well, collapse on the second page with non-zero terms on the lowest row, that is $^2\E_{ver}^{-i,-j}\neq 0$ if  $i\neq r$.  The convergens of the spectral sequence implies that $H^{r-i}_{\bt}(\mathcal{S}_I)\simeq ~  ^2\E_{ver}^{-i,-r}$. Since $^2\E_{ver}^{-i,-r}$ is a subquotient of $H^r_{\bt}(\mathcal{Z}_i)=Z_i(\ff)\otimes_RH^r_{\bt}(S(-i))$, we have $\End(H^{r-i}_{\bt}(\mathcal{S}_I))\leq \End(Z_i(\ff)\otimes_RH^r_{\bt}(S(-i)))=-r+i$.  Accordingly, $\Reg(\mathcal{S}_I)=\max\{\End((H^{r-i}_{\bt}(\mathcal{S}_I))+r-i: i\geq 0\}\leq 0$.  On the other hand, since the ideal definition of $\mathcal{S}_I$ is generated in degree $1$, its Castelnuovo-Mumford regularity must exceed $0$, see \cite[Theorem 16.3.1]{localCohomology}. Therefore, $\Reg(\mathcal{S}_I)=0$. 
 
 $(2)$ For this part we apply \cite[Theorem 3.3.18]{vasconcelos1994arithmetic}. The latter explains the result for a complex $\mathcal{Z}(E)$ where $E$ is an $R$-module.  However, in the case where $E$ is an ideal say $E=I$, then the complex $\mathcal{Z}(E)$ is the same as the approximation complex $\mathcal{Z}(I)$, see \cite[Lemma 3.3(a)]{Onthearithmetic}. Since $I$ is generated by a proper sequence, $\mathcal{Z}(I)$ is acyclic. Hence, by \cite[Theorem 3.3.18]{vasconcelos1994arithmetic}, $$\beta^S_i(\mathcal{S}_I)=\sum_j\beta^R_{i-j}(Z_j(\ff)) $$ that implies the assertion. 
 \end{proof}
%%%%% THeorem  
\begin{Theorem}\label{residualsliding} Let $R$ be a Noetherian local ring with infinite residue field and $I=(f_1,\ldots,f_r)$ be generated by a proper sequence. Assume that $R$ satisfies the Serre's condition S$_{s}$ and   $\pdim_R(Z_{i}(\ff))\leq r-i-1$ for all $i\geq 1$.  Then any $s$-residual intersection of $I$ admits a free approach. 

Furthermore, if  $R$ satisfies the Serre's condition S$_{s+1}$, then any arithmetic $s$-residual intersection of $I$ is unmixed of finite projective dimension. 
\end{Theorem}
\begin{proof}
The proof of this theorem is similar to the one in the previous section. However, we have to modify the resolutions. Since $I$ is generated by a proper sequence  and   $\pdim_R(Z_{i}(\ff))\leq r-i-1$ for all $i\geq 1$, \autoref{propersequence} implies that $\Reg_S(\mathcal{S}_I)=0$ and $\pdim_S(\mathcal{S}_I)\leq r-1$. Therefore the minimal free resolution of $\mathcal{S}_I$ as an $S$-module is of the form 
\begin{equation}\label{resSI}
P_{\bullet}:0\longrightarrow P_{r-1}\longrightarrow\ldots\longrightarrow P_{2}\lra P_{1} \lra S
\end{equation}
where $P_i=\oplus_{j\leq i} S(-j)$.

The complex $L_{\bullet}$ in \autoref{Lcomplex} is then replaced with $P_{\bullet}$ above. Set $G'_{\bullet}:=\Tot(P_{\bullet}\otimes _{S}K_{\bullet}(\gamma;S))$. We have, $H_{i}(G'_{\bullet})$ are $(\gamma+\mathcal{L})$-torsion modules.  The desired free resolution is then 
$Q_{\bullet}=H^{r}_{\bt}(G'_{\bullet})_{[0]}$ augmented with the transgression map coming from the spectral sequences driven from the double complex $(C^{\bullet}_{\bt}\otimes _{S}G'_{\bullet})_{[0]}$. More precisely, 
\begin{equation}\label{Q}
Q_{\bullet}:0\lra Q_{s}\lra\ldots\lra Q_{2}\xra{d_{2}}Q_{1}\xra{d_{1}}R
\end{equation}
where the map $d_{1}$ is constructuted from a spectral sequence similar to \autoref{bigdouble}. Setting $\kappa=\im(d_{1})$. The same argument as in \cite[Theorem 2.11]{hassanzadeh2012cohen} shows that 
\begin{itemize}
\item{$Q_{\bullet}$ is acyclic; }
\item{$\V(\kappa)=\V(J)$;}
\item{$\kappa_{\fp}=J_{\fp}=\fa_{\fp}$ for all $\fp\in \Spec(R)\setminus \V(I)$;}
\end{itemize}
Moreover, if $R$ satisfies the Serre's condition S$_{s+1}$, for all $\fp$ with $\Ht(\fp)\geq s+1$, $\depth(R/\kappa)_{\fp}>0$. That is $\kappa$ is an unmixed ideal of height $s$. 
The equality between $\kappa$ and $J$ happens in the case of arithmetic residual intersection, as it is similarly proven in  \cite[Theorem 2.11]{hassanzadeh2012cohen}. 
\end{proof}
Once more, we go back to our working example.
\begin{Example}\label{Example5}  Let $R$ and $I$ be the ring and ideal in \autoref{Ex2}(ii). Then $I=I_{2}(H)\subset R$  is an ideal of grade $2$. $I=(f_{1},f_{2},f_{3})$ is an almost complete intersection hence generated by a proper sequence.  Since $\grade(I)=2$, Hilbert-Burch theorem implies that $Z_{1}\simeq R^{2}$,  a free module.  Since $R$ is S$_{3}$ and hence S$_{2}$,  \autoref{residualsliding} implies that any $2$-residual intersection  $J=\fa:I$ admits free approach.  For instance, let $M$ be a general $3\times 2$ matrix of linear forms 
say
$$M=\left(
\begin{matrix} 
 x_1+x_2+x_3+x_5+x_6&x_1+x_4+x_6\\
 x_1+x_2+x_4&x_2+x_3+x_5+x_6\\
 x_2+x_4+x_5& x_1+x_2+x_5+x_6\\
\end{matrix}
\right)
$$
and  $\fa$ the ideal generated by the entries of $(f_{1},f_{2},f_{3})\cdot M$. 
Computation using Macaulay2 shows that $\Ht(I + J) = 3$. Consequently, $J$ is a geometric $2$-residual intersection of $I$. Since $R$ satisfies $S_3$, \autoref{residualsliding} implies that $J$ is unmixed of codimension $2$ and has a finite projective dimension of $2$. Notably, in this example, $\depth(R/I) = \dim(R/I) - 1$. Thus, the standard mapping cone technique in linkage theory does not suffice to construct a free resolution for $J$.

Now, we show that $I$ admits $3$-residual intersections. Indeed any almost complete intersection generated by $r$ elements admits an $r$-residual intersection.
To see this, we may reduce by a regular sequence $a_{1},\ldots, a_{r-1}$ inside $I$ and suppose that   $I$ is a principal ideal say $I=(f)$. Then $\dim(R)\geq 1$. Let $m$ be an element of the maximal ideal that does not belong to the minimal primes of $R$. Setting $a_{0}=mf$, we have $(m)\subseteq ((a_{0}):I)$. Hence $\Ht((a_{0}):I)\geq 1$. Lifting to the original ring, we have $\Ht((a_{0},a_{1},\ldots,a_{r-1}):I)\geq r$. Moreover $(a_{0},a_{1},\ldots,a_{r-1})\neq I$, because $a_{0}$ is not a minimal generator of $I$ and $\mu(I)=r$.

We saw in \autoref{Ex2}(ii)  that, moreover, $I$ is not $3$-minimally generated ideal from height $3$. Hence to study $3$-residual intersections of $I$, we can apply \autoref{residualsliding} but not \autoref{residual1}.  Accordingly, any $3$-residual intersection of $I$ admits a free approach. Computation by Macaulay2 shows that, in this example, a general $3$-residual intersection of $I$ is of finite projective dimension $3$ and hence $\depth(R/J)=0$.
\end{Example}
\autoref{Examle4} is a case where one can use \autoref{residual1} but not \autoref{residualsliding} . 

%%% COrollary
We now, recall the definition of the condition ``sliding depth''. 
\begin{Definition}\label{sliding} Keep the notation at the beginning of this section and assume that $R$ is a local ring.  The ideal $I$ is said to satisfy  the sliding depth condition $\SD_{k}$, if $\depth(H_{i})\geq d-r+i+k$ for all  $i\leq r-g$. The ideal $I$ satisfies the sliding depth condition on cycles, $\SDC_{k}$, if  $\depth(Z_{i})\geq d-r+i+k$ for all $ i\leq r-g$. 
If $k=0$, we only say that $I$ satisfies $\SD$ or $\SDC$.
\end{Definition}
For Cohen-Macaulay local ring, $\SD_{k}$ implies $\SDC_{k+1}$ for all $k$, if $g\geq 2$ then $\SDC_{k+1}$ implies $\SD_{k}$; moreover $\SD$ is equivalent to $\SDC_{1}$ for $g\geq 1$; see \cite[Proposition 2.4]{HamidNaeliton}.

When $R$ is a regular local ring, the condition $\pdim(Z_{i})\leq r-i-1$ for all $i\geq 1$ is equivalent to $\depth(Z_{i})\geq d-r+i+1$ for all $i\geq 1$. The later is the $\SDC_{1}$ condition. Thus we have the following corollary.

\begin{Corollary}\label{regularring} Let $R$ be regular local  and $I$ satisfies $\SD$. Then any algebraic residual intersection of $I$ admits a free approach.  Moreover, any arithmetic $s$-residual intersection of $I$ is Cohen-Macaulay.
\end{Corollary}

%% remark
\begin{Remark}\label{questionhuneke}One of the open questions in the theory of residual intersections is whether "in a Cohen-Macaulay local ring, are the algebraic residual intersections of ideals with sliding depth Cohen-Macaulay?" \cite[Question 5.7]{huneke1988residual}. This question is addressed in \cite[Theorem 2.11]{hassanzadeh2012cohen}, which shows that "arithmetic residual intersections of ideals with \(\SD\) are Cohen-Macaulay." Furthermore, \cite[Corollary 5.2]{boucca2019residual} and \cite[Theorem 4.5]{ChardinNaeliton} demonstrate that "in a Cohen-Macaulay local ring, residual intersections of ideals with \(\SD_{1}\) are Cohen-Macaulay." The question for ideals with \(\SD\) remains open.

Another point is that unlike the case of residual intersections of $r$-minimally generated ideals,  the structure of the complex $Q_{\bullet}$ in  \autoref{Q} of \autoref{residualsliding} does not provide a uniform upper bound for the multiplicity of residual intersections. The difference between complexes \(Q_{\bullet}\) and \(F_{\bullet}\) is that the terms of the latter are constructed using the Koszul complex. In contrast, the terms of \(Q_{\bullet}\) are built using the free resolution of the cycles of the Koszul complex, thus depending heavily on the structure of \(I\).

\end{Remark}

An immediate corollary of \autoref{residualsliding}, \autoref{residual1} and \autoref{Pcohdim} is the following
\begin{Corollary}\label{residualcohdim} Let $R$ be a Cohen-Macaulay local ring of positive characteristic and  $J=\fa:I$ an $s$-residual intersection. Assume either $I$ is $r$-minimally generated from height $s-1$ and $s\geq r$; or $I$ satisfies G$_s^-$ condition \footnote {$\mu(I_\fp)\leq \Ht(\fp)+1$ for any prime $\fp\supseteq I$ with $\Ht(\fp)< s$}  and $\pdim(Z_{i})\leq r-i-1$ for all Koszul cycles of $I$ and all $i$. Then
\begin{itemize}
 \item $J$ is cohomologically complete intersection, and 
 \item {assume that  $\widehat{R}$ is connected in dimension $d-1$.  If  $J$ has at least two minimal primes. Then  for any non-empty partition of $\min(J)$, say $A\cup B=\min(J)$, we have $\Ht(\bigcap_{\fp\in A}\fp+\bigcap_{\fq\in B}\fq)=s+1.$}
 \end{itemize}
\end{Corollary}

\autoref{residualcohdim} does not hold in characteristic zero. Bruns-Schwanzl \cite[Corollary to Lemma 2]{BrunsSchwanzl} show that over a field of characteristic zero, the ideal of $t\times t$ minors of a generic $m\times n$ matrix has cohomological dimension $mn-t^{2}+1$.  See also \cite[Example 6.1]{WalterAlgorithm} for the case where $m=2$ and $n=3$. 

Consider the generic matrix $H=\left(\begin{matrix}x_{1}&x_{2}&x_{3}\\x_{4}&x_{5}&x_{6}\\\end{matrix}\right)$ in the ring $\mathbb{Q}[x_{1},\ldots,x_{6}]$. Set $J=I_{2}(H)$. Then $J=(\Delta_{1},\Delta_{2}):(x_{1},x_{4})$
where $\Delta_{1}$ and $\Delta_{2}$ are minors of $H$ obtained by removing the first column. It is clear that the ideal $I=(x_{1},x_{4})$ satisfies the conditions of \autoref{residualcohdim}, however $H^{i}_{J}(R)\neq 0$ for $i=2,3$. 

%\vspace{.7cm}
%{\bf Acknowledgement.}  The author would like to thank Ehsan Tavanfar for reading the first manuscript of this paper and providing valuable suggestions for improving its presentation. His observations, particularly regarding the content of \autoref{Pcohdim}, significantly enhanced the clarity and quality of the work.

\section{The structure of $\tau$}\label{taustructure}
In the paper \cite{boucca2019residual}, the authors determined the structure of $s$-residual intersections $J=\fa:I$ for a large class of ideal $I$. They show that residual intersections are Koszul-Fitting ideals. We will see that the ideal $\tau$ has, as well, a Koszul-Fitting structure. For that we recall the definition of Kitt-ideals. 

 Let $R$ be a ring and $\fa \subseteq I$ two finitely generated ideals of $R$. Consider $\mathbf{f}=\:f_1,\dots,f_r$ and $\mathbf{a}=\:a_1,\dots,a_s$ system of generators of $I$ and $\fa$, respectively. Let $\Phi=[c_{ij}]$ be an $r\times s$ matrix in $R$ such that $[\mathbf{a}]=[\mathbf{f}]\cdot \Phi$.
Let  $K_\bullet(\ff;R)= R\langle e_1,\dots,e_r\:;\:\partial(e_i)=f_i\rangle$ be the Koszul complex equipped with the structure of differential graded algebra. Let 
$\zeta_j = \sum_{i=1}^r c_{ij}e_i$ for $1\leq j\leq s$, $\Gamma_\bullet=R\langle\zeta_1,\dots ,\zeta_s\rangle$ the algebra generated by the $\zeta$'s, and $Z_\bullet = Z_{\bullet}(\ff;R)$ be the algebra of Koszul cycles. Looking at the elements of degree $r$ in the sub-algebra of $K_{\bullet}$ generated by the product of $\Gamma_\bullet$ and $ Z_\bullet$, one defines 
\begin{equation}\label{kitt}
\Kitt(\fa,I):=\langle\Gamma_\bullet \cdot Z_\bullet\rangle _r.
\end{equation}
In \cite[Theorem 2.3]{hassanzadeh2024deformation} and  \cite{boucca2019residual}, the authors present a list of properties of $\Kitt$ ideals and their relation with $s$-residual intersections.  Some of the main properties are the following: The ideal $\Kitt(\mathfrak a,I)$ does not depend on the choice of generators of $\fa$ and $I$ or the representative matrix, $\Kitt(\mathfrak a,I)= \fa+\langle\Gamma_\bullet \cdot \tilde{H}_\bullet\rangle_r$ where $\tilde{H}_\bullet$ is the sub-algebra of $K_{\bullet}$ generated by the representatives of Koszul homologies, $\Kitt(\fa,I)\subseteq J$  have the same radical and they coincide whenever $J$ is an $s$-residual intersection of $I$ and, $s\leq \grade(I)+1$  or $R$ is Cohen-Macaulay and  $I$ satisfies $\SD_1$. 

The idea of the following definition comes from a conversation with Laurent Busé in Luminy 2023. 
\begin{Definition} The {\bf Kitt filtration} of the quotient ideal $J$ is the sequence $\Kitt_0(\fa,I)\subseteq\cdots\subseteq \Kitt_r(\fa,I) \subseteq J$ where $$\Kitt_i(\fa,I)=\langle\Gamma_\bullet \cdot \langle Z_0,\ldots,Z_i\rangle\rangle_r$$
\end{Definition}
$\Kitt_0(\fa,I)$  is the elements of degree $r$ in the exterior algebra $\Gamma_\bullet$, hence $\Kitt_0(\fa,I)=I_r(\Phi)$ where $[\mathbf{a}]=[\mathbf{f}]\cdot \Phi$,  $\Kitt_1(\fa,I)=\Fitt_0(I/\fa)$, and $\Kitt_r(\fa,I)=\Kitt(\fa,I)$, see \cite[Proposition 20.7]{EisenbudBook} and  \cite[Theorem 4.19]{boucca2019residual}. Accordingly, $\sqrt{J}=\sqrt{\Kitt_1(\fa,I)}=\sqrt{\Kitt(\fa,I)}$. 

Now, recall the definition of the ideal $\tau$ from \autoref{F}, we have
\begin{Theorem}\label{taustr} Let $R$ be a Noetherian ring and $\fa \subseteq I$ two ideals of $R$. Keeping the above notation, we have $\tau= \fa+\Kitt_0(\fa,I)$.  In particular, $\mu(\tau)\leq s+\binomial{s}{r}$.
\end{Theorem}
\begin{proof} 
Recall that $I=(f_1,\ldots,f_r)$ and $B_i$ is  the module of $i$-th boundaries of its Koszul complex, i.e. the image of $\delta_{i+1}^{\ff}$. The approximation $\mathcal{B}_\bullet$-complex (respectively $\mathcal{M}_\bullet$-complex), again introduced by Herzog-Simis-Vasconcelos, is defined similarly as the $\mathcal{Z}_\bullet$-complex by replacing $Z_i$ with  $B_i$ (respectively, $H_i$) in the definition for each $i\ge 1$ (respectively, $i\ge 0$). The $\mathcal{B}_\bullet$-complex (respectively, $\mathcal{M}_\bullet$-complex)  of $I$ is denoted by $\mathcal{B}_\bullet=\mathcal{B}_\bullet(f;R)$ (respectively, $\mathcal{M}_\bullet:=\mathcal{M}_\bullet(f;R)$).
We consider the following double complex 
\begin{equation}\label{black}
 \xymatrix{
         &                           &                                                                               &             & K_{r}\otimes S(-1) \ar[d] &                       \\
         &                           &                                                                               &0\ar[r]      &K_{r-1}\otimes S(-1) \ar[d] &    \\
         &                           &  0 \ar[d]                                                                     &      &\vdots \ar[d]  &                   \\
         &  0 \ar[r] \ar[d]          & K_r\otimes S(-r+2)\ar[r] \ar^{\delta_{r}^f \otimes 1}[d] &\ldots\ar[r] &K_{3}\otimes S(-1) \ar^{\delta_{3}^f \otimes 1}[d]& \\
 0 \ar[r]& K_r\otimes S(-r+1) \ar^{\bar \delta^{T}_{r}}[r]\ar@{>>}^{\bar \delta_{r}^f \otimes 1}[d] & K_{r-1}\otimes S(-r+2) \ar^(.7){\bar \delta^{T}_{r-1}}[r] \ar@{>>}^{\bar \delta_{r-1}^f \otimes 1}[d]&\ldots\ar^{\bar \delta^{T}_{3}}[r] &K_{2}\otimes S(-1)\ar@{>>}^{\bar \delta_{2}^f \otimes 1}[d]&        \\
 \textcolor{orange}{0}\ar@[yellow][r]& \textcolor{orange}{B_{r-1}\bigotimes_{R}S(-r+1)}\ar@[yellow]^{ \delta^{T}_{r-1}}[r]& \textcolor{orange}{B_{r-2}\otimes S(-r+2)} \ar@[yellow]^(.7){ \delta^{T}_{r-2}}[r] & \ldots\ar@[yellow]^{\delta^{T}_{2}}[r] & \textcolor{orange}{B_1\bigotimes_{R}S(-1)}\ar@[yellow]^{\delta^{T}_{1}}[r]&  \textcolor{orange}{R\bigotimes_{R}S}\\ 
}
\end{equation}

 Totalizing the black part and adding the composed map  as we did in \autoref{Lcomplex},  we get the following diagram with exact columns

 \begin{equation}\label{zlb}
\xymatrix{
         & 0 \ar[d] &  & 0 \ar[d] & 0 \ar[d] & \\
0 \ar[r] & K_r \otimes \bigoplus_{i=1}^{r-2} S(-i) \ar[r] \ar[d] & \ldots \ar[r] & Z_{3}\otimes_R S(-2) \oplus K_3 \otimes S(-1) \ar[r] \ar[d] & Z_{2} \otimes_R S(-1) \ar[d] \ar[rr] & & 0 \ar[d] \\
0 \ar[r] & L_{r-1} \ar[r] \ar[d] & \ldots \ar[r] & L_{2} \ar[r] \ar[d] & L_{1} \ar[d]^{\bar \delta_{2}^f \otimes 1} \ar[rr]^{\delta^{T}_{1} \circ (\bar \delta_{2}^f \otimes 1)} & & S \ar@{=}[d] \\
0 \ar[r] & B_{r-1} \otimes_R S(-r+1) \ar[r] \ar[d] & \ldots \ar[r] & B_{2} \otimes_R S(-2) \ar[r] \ar[d] & B_{1} \otimes_R S(-1) \ar[d] \ar[rr]^{\delta^{T}_{1}} & & S \\
         & 0 &  & 0 & 0 & 
}
\end{equation}
 We then tensor each row of this complex with the Koszul complex  $K(\gamma_{1},\ldots,\gamma_{s};S)$ and take the total complex of the resulting double complex for each row. This process results in the following diagram, where the columns remain exact. The structure of  $G_{\bullet}$ has been explained in \autoref{G} and those of $G'_{\bullet}$  and $_BG_{\bullet}$ are likewise. 
 \begin{equation}\label{Gzlb}
\xymatrix{
        & 0 \ar[d] &  & 0 \ar[d] & 0 \ar[d] & \\
G'_{\bullet}\ar[d]^{\alpha_{\bullet}}:0 \ar[r] & G'_{r+s-1}\ar[d]\ar[r] & \ldots \ar[r] & G'_2\ar[r]\ar[d] & G'_1 \ar[d] \ar[rr] & & 0 \ar[d] \\
G_{\bullet}\ar[d]^{\beta_{\bullet}}:0 \ar[r] & G_{r+s-1} \ar[r] \ar[d] & \ldots \ar[r] & G_{2} \ar[r] \ar[d] & G_{1} \ar[d] \ar[rr]& & S \ar@{=}[d] \\
_BG_{\bullet}:0 \ar[r] & _BG_{r+s-1}  \ar[r] \ar[d] & \ldots \ar[r] & _BG_{2} \ar[r] \ar[d] & _BG_{1}  \ar[d] \ar[rr] & & S \\
         & 0 &  & 0 & 0 & 
}
\end{equation}

The chain maps $\alpha_{\bullet}$ and $\beta_{\bullet}$ induce a chain map between double complexes $G'_{\bullet}\otimes_S C^{\bullet}_{\bt}$, $G_{\bullet}\otimes_S C^{\bullet}_{\bt}$ and $_BG_{\bullet}\otimes_S C^{\bullet}_{\bt}$.  The module $\mathcal{F}_1$ introduced in \autoref{F1} is the totalization of each of these double complex by deleting the first row. Hence the  chain maps $\alpha_{\bullet}\otimes \Id$ and $\beta_{\bullet}\otimes \Id$   induce the inclusions
 \begin{equation}
 \alpha_{\bullet}\otimes \Id(\mathcal{F}_1(G'_{\bullet}\otimes_S C^{\bullet}_{\bt}))\subseteq \mathcal{F}_1(G_{\bullet}\otimes_S C^{\bullet}_{\bt}){\text~ and~} \beta_{\bullet}\otimes \Id(\mathcal{F}_1(G_{\bullet}\otimes_S C^{\bullet}_{\bt}))\subseteq \mathcal{F}_1(_BG_{\bullet}\otimes_S C^{\bullet}_{\bt}).
 \end{equation}
These inclusions show that the blue diagram in the follwoing picture is commutative.
 \begin{equation}\label{3D}
    % Change the values of theta and phi to rotate the picture
\tdplotsetmaincoords{60}{0}  % {theta}{phi}
  \begin{tikzpicture}[tdplot_main_coords]

% Define the coordinates of the vertices of the parallelograms
% Bottom surface (stretched more horizontally)
\coordinate (A1)at (6, 0, 0);  
\coordinate (B1) at (-3,0, 0);
\coordinate (C1) at (-1, 4, 0);
\coordinate (D1) at (8, 4, 0);
\coordinate (G1) at (3, 0, 0);
\coordinate (H1) at (0, 0, 0);
% Second surface (stretched more horizontally)
\coordinate (A2) at (6, 0, 3);  
\coordinate (B2) at (-3, 0, 3);
\coordinate (C2) at (-1,4, 3);
\coordinate (D2) at (8, 4, 3);
\coordinate (G2) at (3, 0, 3);
\coordinate (H2) at (0, 0, 3);
% Top Surface
\coordinate (A3) at (6, 0, 6);  
\coordinate (B3) at (-3,0, 6);
\coordinate (C3) at (-1, 4, 6);
\coordinate (D3) at (8, 4, 6);
\coordinate (G3) at (3, 0, 6);
\coordinate (H3) at (0, 0, 6);
%Fill Points
\fill[black] (D1) circle (1pt);
\fill[black] (D2) circle (1pt);
\fill[black] (D3) circle (1pt);
\fill[black] (G1) circle (2pt);
\fill[black] (G2) circle (2pt);
\fill[black] (G3) circle (2pt);
\fill[black] (B1) circle (3pt);
\fill[black] (B2) circle (3pt);
\fill[black] (B3) circle (3pt);
\fill[black] (H1) circle (3pt);
\fill[black] (H2) circle (3pt);
\fill[black] (H3) circle (3pt);
% Draw the bottom parallelogram
\draw[thick] (A1) -- (B1) -- (C1) -- (D1) -- cycle;

% Draw the top parallelogram
\draw[thick] (A2) -- (B2) -- (C2) -- (D2) -- cycle;
% Draw the top parallelogram
\draw[thick] (A3) -- (B3) -- (C3) -- (D3) -- cycle;
% Connect corresponding vertices
\draw[->,thick, blue] (G1) -- (D1);
\draw[->,thick, blue] (G2) -- (D2);
\draw[->,thick, blue] (G3) -- (D3);
\draw[->>,thick, blue] (G2) -- (G1) node[midway, right] {$H^r_{\mathbf{t}}(\beta_r)_{[0]}$};
\draw[->,thick, blue] (G3) -- (G2) node[midway, right] {$H^r_{\mathbf{t}}(\beta_r)_{[0]}$};
% Add arrows for maps
\draw[ ->,thick, blue] (D2) -- (D1)node[midway, right] {=}; % From top-right vertex of top surface 
\draw[ ->,thick, blue] (D3) -- (D2);
%to bottom surface
\draw[->>, thick] (B2) -- (B1); % From top-right vertex of top surface to bottom surface
\draw[->, thick] (B3) -- (B2);
\draw[->, thick] (H3) -- (H2);
\draw[ ->,thick] (H2) -- (H1);
 % Label the vertices
\node[anchor=south west] at (D3) {$0$};
\node[anchor=south west] at (D2) {$R$};
\node[anchor=north west] at (D1) {$R$};
\node[anchor=north] at (G1) {$H^r_{\mathbf{t}}(_BG_r)_{[0]}$};
\node[anchor=south] at (G2) {$H^r_{\mathbf{t}}(G_r)_{[0]}$};
\node[anchor=south] at (G3) {$H^r_{\mathbf{t}}(G'_r)_{[0]}$};
\node[anchor=south east] at (B3) {$H^r_{\mathbf{t}}(G'_{r+s-1})_{[0]}$};
\node[anchor=south east] at (B2) {$H^r_{\mathbf{t}}(G_{r+s-1})_{[0]}$};
\node[anchor=north east] at (B1) {$H^r_{\mathbf{t}}(_BG_{r+s-1})_{[0]}$};
\end{tikzpicture}
  \end{equation} 
  
Moreover, since $H^{r-1}_{\mathbf{t}}(_BG_r)=0$, for each integer $i$ between $r$ and  $r+s-1$, the column $0\ra H^r_{\mathbf{t}}(G'_i)_{[0]}\ra H^r_{\mathbf{t}}(G_i)_{[0]}\ra H^r_{\mathbf{t}}(_BG_i)_{[0]}\ra 0$ is an exact sequence. Clearly, $0\ra R\ra R\ra 0$ is also exact. Hence the Snake lemma implies that the zeroth homology of the middle complex is the same as the zeroth homology of the complex in the bottom.

The complex in the bottom is analogous to the $\mathcal{Z}^+_{\bullet}$-complex introduced in \cite{boucca2019residual}.  
By replacing  $\mathcal{Z}_\bullet$-complex with 	 $\mathcal{B}_\bullet$-complex in the construction of $\mathcal{Z}_\bullet^+$ and following the same procedure as in the construction of $\mathcal{Z}^+_\bullet$ in  \cite{boucca2019residual}, we obtain a ``boundary" residual approximation complex denoted by $\mathcal{B}^+_\bullet$. $\mathcal{B}^+_\bullet$ is the complex $H^r_{\mathbf{t}}(_BG_{\bullet})_{[0]}$ augmented with the lowest diagonal map in \autoref{3D}. 
Set $\tau_{\mathcal{B}}$ the image of $\partial^{\mathcal{B}^+_\bullet}_1:\mathcal{B}_1^+\rightarrow \mathcal{B}^+_0=R$.   The equality $H_0(F)=H_0(\mathcal{B}^+_\bullet)$ implies that $\tau=\tau_B$. 

To determine the structure of $\tau_B$, we consider the Koszul DG-algebra $K_\bullet(\ff;R)= R\langle e_1,\dots,e_r\:;\:\partial(e_i)=f_i\rangle$  and the DG-ideal $B_{\bullet}=\oplus_{i=0}^{r-2} B_i$. We extend  $B_{\bullet}$ to the algebra $B'_{\bullet}=R\oplus_{i=1}^{r-2} B_i$. Adopting the structure of the Kitt-ideal for the complex $\mathcal{B}^+_{\bullet}$, \cite[Theorem 4.9]{boucca2019residual} becomes 
$$\tau_B=\langle\Gamma_\bullet \cdot B'_\bullet\rangle _r.$$
On the other hand,  \cite[Lemma 4.2]{boucca2019residual} shows that $\fa=\langle\Gamma_\bullet \cdot B_\bullet\rangle _r.$ Therefore, 
$$\tau=\tau_B=\fa+ R\cdot\wedge^r \Gamma_\bullet=\fa+I_r(\Phi)=\fa+\Kitt_0(\fa,I)$$
as desired.
\end{proof}
As an immediate corollary of \autoref{taustr} and \autoref{radical}, we have
\begin{Corollary} Let $R$ be a Noetherian ring of dimension
$d$ and $I=(f_1,\ldots,f_r)$ contains a regular element. Let $J=(a_1,\cdots,a_s):I$ with  $\Ht(J)=\alpha$. Suppose that $s\geq r$ and  $I$ is $r$-minimally generated from height $\alpha-1$. Then the arithmetic rank of $J$ is at most $s+\binomial{s}{r}$.
\end{Corollary}

\vspace{.7cm}
{\bf Acknowledgement.}  The author would like to thank Ehsan Tavanfar for reading the first manuscript of this paper and providing valuable suggestions for improving its presentation. His observations, particularly regarding the content of \autoref{Pcohdim} and \autoref{taustructure}, significantly enhanced the clarity and quality of the work. He also extends his gratitude to the referee for their careful reading of the paper.

\bibliographystyle{alpha}
\bibliography{References}

\end{document}